\documentclass[1 [leqno,11pt]{amsart}
\usepackage{amssymb, amsmath,amsmath,latexsym, amssymb, amsfonts, amsbsy, amsthm, mathtools, graphicx, CJKutf8,CJKnumb, CJKulem, color}
\usepackage{comment}
\usepackage{tikz}
\tikzstyle{startstop} = [rectangle, rounded corners, minimum width=1cm, minimum height=1cm,text centered, draw=black]
\tikzstyle{io} = [rectangle, rounded corners, minimum width=1cm, minimum height=1cm,text centered, draw=black]
\tikzstyle{process} = [rectangle, rounded corners, minimum width=1cm, minimum height=1cm,text centered, draw=black]
\tikzstyle{decision} = [rectangle, rounded corners, minimum width=1cm, minimum height=1cm,text centered, draw=black]
\tikzstyle{explain} = [minimum width=1cm, minimum height=0.5cm,text centered, draw=black]
\tikzstyle{equations} = []
\tikzstyle{arrow} = [thick,->,>=stealth]
\tikzstyle{d-arrow} = [thick,->,dashed,>=stealth]
\usetikzlibrary{shapes.geometric, arrows,plotmarks,backgrounds}
\usepackage{pgfplots}
\usepackage{tikz}
\usepackage{hyperref}
\usepackage{cancel}
\numberwithin{equation}{section}

\allowdisplaybreaks[4]

\let\al=\alpha

\let\d=\delta

\let\ep=\epsilon

\let\f=\frac

\let\Om=\Omega

\let\th=\theta
\let\pa=\partial

\let\pa=\partial

\usepackage{scalerel,stackengine}
\stackMath
\newcommand\reallywidehat[1]{%
\savestack{\tmpbox}{\stretchto{%
  \scaleto{%
    \scalerel*[\widthof{\ensuremath{#1}}]{\kern-.6pt\bigwedge\kern-.6pt}%
    {\rule[-\textheight/2]{1ex}{\textheight}}
  }{\textheight}%
}{0.5ex}}%
\stackon[1pt]{#1}{\tmpbox}%
}

\def\dive{\mathop{\rm div}\nolimits}

\newcommand{\beq}{\begin{equation}}
\newcommand{\eeq}{\end{equation}}
\newcommand{\ben}{\begin{eqnarray}}
\newcommand{\een}{\end{eqnarray}}
\newcommand{\beno}{\begin{eqnarray*}}
\newcommand{\eeno}{\end{eqnarray*}}

\newtheorem{theorem}{Theorem}[section]

\newtheorem{lemma}[theorem]{Lemma}
\newtheorem{proposition}[theorem]{Proposition}

\newtheorem{remark}[theorem]{Remark}

\begin{document}

\title[]{Improved stability threshold of the Two-Dimensional Couette flow for Navier-Stokes-Boussinesq Systems via quasi-linearization}

\author{Binqian Niu}
\address[B. Niu]{
School of Mathematical Science, Shanghai Jiao Tong University, Shanghai, 200240, P.R.China.}
\email{bqbling@sjtu.edu.cn}
\author{Weiren Zhao}
\address[W. Zhao]{Department of Mathematics, New York University Abu Dhabi, Saadiyat Island, P.O. Box 129188, Abu Dhabi, United Arab Emirates.}
\email{zjzjzwr@126.com, wz19@nyu.edu}
\maketitle

\begin{abstract}
In this paper, we improve the size requirement of the perturbations for the asymptotic stability of the Couette flow in stratified fluids governed by the two-dimensional Navier-Stokes-Boussinesq system. More precisely, the size of perturbed temperature is improved to $\nu^{2/3}$ from $\nu^{5/6}$ in the paper of Zhang and Zi [J. Math. Pure. Anal. 179:123-182 (2023)]. The idea is the quasi-linearization. The main system is decomposed into two or more equations: a good equation (might be linear) that carries the regularity and size of the initial data, and some quasi-linear and nonlinear equations that contain the nonlinear part, which start from zero initial data. 
\end{abstract}

\section{Introduction}
In this paper, we consider the two-dimensional Navier-Stokes-Boussinesq system on the infinite periodic channel $\mathbb T\times \mathbb R$:
\begin{equation}\label{eq:1}
    \begin{cases}
        \pa_tv+v\cdot\nabla  v+\nabla P-\nu\Delta v=-\varrho e_2,\\
        \pa_t\varrho+v\cdot\nabla\varrho-\nu\Delta\varrho=0,\quad \dive\, v=0,\\
        v(0,x,y)=v_{\mathrm{in}}(x,y),\quad \varrho(0,x,y)=\varrho_{\mathrm{in}}(x,y),
    \end{cases}
\end{equation}
where $v=(v^1(t,x,y),v^2(t,x,y))$ is the divergence-free velocity field, $\varrho=\varrho(t,x,y)$ is the temperature, and $P=P(t,x,y)$ is the pressure. Moreover, $e_2=(0,1)^{\mathsf{T}}$ is the unit vector in the vertical direction. Here the small constant $\nu$ is the viscosity coefficient and the thermal diffusivity. 

For any constants $\rho_0,p_0\in\mathbb R$, 
$
    v_s=(y,0)^{\mathsf{T}},\  \varrho_s=\rho_0,\  P_s=\rho_0y+p_0
$
is a class of steady solutions to system \eqref{eq:1}. In this paper, we study the asymptotic stability of these steady solutions. We now introduce the perturbations
\begin{align*}
    u=v-(y,0)^{\mathsf{T}}, \quad \th=\varrho-\rho_0,\quad p=P-(\rho_0y+p_0).
\end{align*}
Then $(u,\th,p)$ satisfies
\begin{equation}
    \begin{cases}
        \pa_tu+y\pa_xu+\begin{pmatrix}
            u^2\\
            0
        \end{pmatrix}+u\cdot\nabla u+\nabla p-\nu\Delta u=\begin{pmatrix}
            0\\
            -\th
        \end{pmatrix},\\
        \pa_t\th+y\pa_x\th+u\cdot\nabla\th-\nu\Delta\th=0,\quad \mathrm{div}u=0,\\
        u(0,x,y)=u_{\mathrm{in}}(x,y),\quad \th(0,x,y)=\th_{\mathrm{in}}(x,y).
    \end{cases}
\end{equation}
The vorticity $\omega=\pa_xu^2-\pa_yu^1$ and the associated stream function $\psi$ satisfy
\begin{equation}\label{eq:o}
    \begin{cases}
        \pa_t\omega+y\pa_x\omega+u\cdot\nabla\omega-\nu\Delta\omega=\pa_x\th,\\
        \pa_t\th+y\pa_x\th+u\cdot\nabla\th-\nu\Delta\th=0,\\
        u=\nabla^{\perp}\psi=(-\pa_y\psi,\pa_x\psi),\\
        \Delta\psi=\omega=\pa_xu^2-\pa_yu^1.
    \end{cases}
\end{equation}
We study the following stability threshold problem.

{\it{Given norms $\|\cdot\|_{Y_1}$ and $\|\cdot\|_{Y_2}$, find $\al=\al(Y_1,Y_2)$ and $\beta=\beta(Y_1,Y_2)$ such that
\begin{align*}
&\|u_{\mathrm{in}}\|_{Y_1}\leq \nu^{\al}\ \text{and} \ \|\th_{\mathrm{in}}\|_{Y_2}\leq \nu^{\beta}\ \Rightarrow \text{stability};\\
&\|u_{\mathrm{in}}\|_{Y_1}\gg \nu^{\al}\ \text{or} \ \|\th_{\mathrm{in}}\|_{Y_2}\gg \nu^{\beta}\ \Rightarrow \text{instability}.
\end{align*}}}
In \cite{DengWuZhang2021, MasmoudiZhaiZhao2022}, the authors proved the asymptotic stability of the steady state if the initial perturbations $(u_{\mathrm{in}},\th_{\mathrm{in}})$ satisfy 
\begin{align*}
\|u_{\mathrm{in}}\|_{H^2}\leq \epsilon_0\nu^{\f12},\ 
 \|\th_{\mathrm{in}}\|_{H^1}+ \||D_x|^{\f16}\th_{\mathrm{in}}\|_{H^1}\leq \epsilon_1\nu^{\f{11}{12}}, 
\end{align*}
for both the finite channel ($\Om=\mathbb T\times[-1,1] $) and the infinite channel ($\Om=\mathbb T\times \mathbb R$) cases. For the infinite channel case, the threshold obtained in \cite{DengWuZhang2021} was improved by Zhang-Zi \cite{zhang2023stability} to 
\begin{align*}
\|u_{\mathrm{in}}\|_{H^{s+1}}\leq \epsilon_0\nu^{\f13},\ 
 \|\th_{\mathrm{in}}\|_{H^s}+ \nu^{\f16}\||D_x|^{\f13}\th_{\mathrm{in}}\|_{H^s}\leq \epsilon_1\nu^{\f{5}{6}},
\end{align*}
where $s>7$. The main ideas in \cite{zhang2023stability} are 
\begin{itemize}
    \item to apply the time-dependent Fourier multipliers to obtain the inviscid damping and enhanced dissipation for the highest Sobolev norm;
    \item to apply a time-dependent Fourier multiplier introduced in \cite{masmoudizhao2022} to capture the growth from the reaction terms.
\end{itemize}
This paper aims to improve the size requirement in the aforementioned result by introducing the quasi-linearization method. Our main theorem states as follows:
\begin{theorem}\label{main-thm}
    Let $m>5$. There is $\epsilon_1>0$ independent of $\nu$, such that for any $\epsilon_0\in (0, \epsilon_1)$, if $\|u_{\mathrm{in}}\|_{H^{m+1}}\leq \epsilon_0\nu^{\f{1}{3}}$ and $\|\langle \partial_x\rangle \theta_{\mathrm{in}}\|_{H^m}\leq \epsilon_0\nu^{\f23}$, then for all $t\geq 0$, the solutions to \eqref{eq:o} satisfy
    \begin{align*}
        &\|\omega_0\|_{L^2}+\nu^{-\f13}\|\theta_0\|_{L^2}\leq  C\epsilon_0\nu^{\f13},\\
        &\|u^1_{\neq}\|_{L^2}+(1+t)\|u^2\|_{L^2}\leq \f{C\epsilon_0\nu^{\f13}}{1+t} e^{-\delta_0\nu^{\f13}t},\\
        &\|\omega_{\neq}\|_{L^2}+\nu^{-\f13}\|\theta_{\neq}\|_{L^2}\leq C\epsilon_0\nu^{\f13} e^{-\delta_0\nu^{\f13}t},
    \end{align*}
    for some constant $C$ independent of $t, \nu$. Here $f_0=\f{1}{|\mathbb{T}|}\int_{\mathbb{T}}f(t,x,y)dx$ is the zero mode of $f$ and $f_{\neq}=f-f_0$ is the non-zero mode of $f$. 
\end{theorem}

\begin{remark}
    The key difficulty in this problem is the $1/3$ derivative loss caused by the buoyancy effect, namely $\partial_x\theta$. The main idea in the proof contains two aspects: 1. We use the quasi-linearization; 2. We introduce different sizes and regularities in different time regions in the estimates of the error part. 
    
    The quasi-linearization makes this buoyancy term $\partial_x\theta$ smaller, so that even if we use the dissipation to absorb the additional derivative which could lead to a large amplification, the solution is still small enough. We refer to section \ref{sec:ql-method} for more details. We believe that this method could be used to improve many threshold results related to shear flows. 

    Without losing size, the system can not propagate additional derivatives in only the $x$ direction. The key observation is that for the short-time region, we can lose less size and even propagate more derivatives in the $x$ direction, and in the long-time region, the inviscid damping offers us an additional smallness, which allows us to use the dissipation. 
\end{remark}

\begin{remark}
    To simplify the proof, we modify the Fourier multipliers in \cite{weizhang2023nonlinear} to capture the nonlinear growth. One can also apply the linear change of coordinates and apply the Fourier multipliers in \cite{masmoudizhao2022, zhang2023stability}. The modification is discussed in section \ref{sec:ql-method}. 
\end{remark}

\begin{remark}
    The size requirement in this paper seems to be optimal for perturbations in Sobolev spaces, provided that the $\nu^{\f13}$ threshold for Navier-Stokes equation \cite{masmoudizhao2022, weizhang2023nonlinear} is optimal. One may regard the estimates for the solutions to \eqref{eq:i} as strong evidence for the size requirement of perturbed temperature, see Proposition \ref{prop: short-time-i} and Proposition \ref{prop: long-time i}. We also refer to \cite{LiMasmoudiZhao2022asymptotic} for the asymptotic stability of two-dimensional Couette flow for the Navier-Stokes equation with a threshold less than $1/3$. One can also expect a larger size of perturbations in the Gevrey class that can lead to asymptotic stability for the Navier-Stokes-Boussinesq system. 
\end{remark}

\begin{remark}
    We also mention some other recent results related to the stability of Couette flow in stratified fluid with different background densities in different dimensions \cite{bedrossian2023nonlinear,zelati2024stability, Jiaxin-2024-pre, MSZ-2022-ARMA, sinambela2024asymptotic, SZZ-2024-2d, ZhaiZhao2022}. 
\end{remark}

\begin{remark}
    Our proof also works for the case when the thermal diffusivity for the temperature equation is comparably different from the viscosity. 
\end{remark}

\subsection{Notation}
We define the time-dependent elliptic operator for $t\geq 0$,
\begin{align*}
    \Lambda_t^2=1-\pa_x^2-(\pa_y+t\pa_x)^2.
\end{align*}
In terms of its symbol, $\Lambda_t^2(k,\xi)=1+k^2+(\xi+tk)^2$. It is easy to check that for any $b\in\mathbb R$, $\Lambda_t^b(\pa_t+y\pa_x)=(\pa_t+y\pa_x)\Lambda_t^b$, namely, $[\Lambda_t^b, \pa_t+y\pa_x]=0$, and for $s>1$,
\begin{align*}
    \|f(t)\|_{L^{\infty}}\leq C\|\hat{f}(t)\|_{L^1}\leq C\|\Lambda_t^sf(t)\|_{L^2}.
\end{align*}
Due to the different behaviors of the zero mode and nonzero modes, we split the velocity into two parts:
\begin{align*}
    u=P_0u+P_{\neq}u=(u_0^1,0)+(u_{\neq}^1,u_{\neq}^2),
\end{align*}
where
\begin{align*}
    u_0^1=-\pa_y(-\pa_y^2)^{-1}\omega_0,\quad u_{\neq}^1=-\pa_y(-\Delta)^{-1}\omega_{\ne},\quad u_{\neq}^2=\pa_x(-\Delta)^{-1}\omega_{\neq}.
\end{align*}

\subsection{Quasi-linearization method}\label{sec:ql-method}

In this section, we discuss the quasi-linearization method, which has been used in many papers, especially to deal with the boundary problem \cite{chen_Li_Wei_Zhang2018transition, chen2024transition, Weizhang-slip-2024}. In \cite{ ZhaiZhao2022}, the authors use this method to decompose the zero-mode into a small part and a large part with higher regularity. The main idea of quasi-linearization is to decompose the equation into two or more equations: a good equation (might be linear) that carries the regularity and size of the initial data, some quasi-linear and nonlinear equations that contain the smaller nonlinear part, which start from zero initial data.
We decompose system \eqref{eq:o} into the following two equations: Let
\begin{align*}
    \omega=\omega^{\mathrm{i}}+\omega^{\mathrm{e}},\quad \th=\th^{\mathrm{i}}+\th^{\mathrm{e}}.
\end{align*}
Here $(\omega^{\mathrm{i}}, \th^{\mathrm{i}})$ solves
\begin{equation}\label{eq:i}
    \begin{cases}
        \pa_t\omega^{\mathrm{i}}+y\pa_x\omega^{\mathrm{i}}-\nu\Delta\omega^{\mathrm{i}}+u^{\mathrm{i}}\cdot\nabla\omega^{\mathrm{i}}=\pa_x\theta^{\mathrm{i}},\\
        \pa_t\th^{\mathrm{i}}+y\pa_x\th^{\mathrm{i}}+u^{\mathrm{i}, 1}_0\partial_x\theta^{\mathrm{i}}-\nu\Delta\th^{\mathrm{i}}=0,\\
        \pa_tu^{\mathrm{i}, 1}_0-\nu\pa_y^2u^{\mathrm{i}, 1}_0=P_0(u^{\mathrm{i}}_{\neq}\cdot\nabla u_{\neq}^{\mathrm{i}, 1})\\
        \omega^{\mathrm{i}}(0,x,y)=\omega_{\mathrm{in}}(x,y),\quad \th^{\mathrm{i}}(0,x,y)=\th_{\mathrm{in}}(x,y).
    \end{cases}
\end{equation}
And $(\omega^{\mathrm{e}}, \th^{\mathrm{e}})$ solves
\begin{equation}\label{eq:e}
    \begin{cases}
        \pa_t\omega^{\mathrm{e}}+y\pa_x\omega^{\mathrm{e}}-\nu\Delta\omega^{\mathrm{e}}+u^{\mathrm{i}}\cdot\nabla\omega^{\mathrm{e}}+u^{\mathrm{e}}\cdot\nabla\omega^{\mathrm{i}}+u^{\mathrm{e}}\cdot\nabla\omega^{\mathrm{e}}=\pa_x\th^{\mathrm{e}},\\
        \pa_t\th^{\mathrm{e}}+y\pa_x\th^{\mathrm{e}}-\nu\Delta\th^{\mathrm{e}}+u_{\neq}^{\mathrm{i}}\cdot\nabla \th^{\mathrm{i}}+u^{\mathrm{i}}\cdot\nabla\th^{\mathrm{e}}+u^{\mathrm{e}}\cdot\nabla\th^{\mathrm{e}}+u^{\mathrm{e}}\cdot\nabla\th^{\mathrm{i}}=0,\\
        \omega^{\mathrm{e}}(0,x,y)=0,\quad \th^{\mathrm{e}}(0,x,y)=0.
    \end{cases}
\end{equation}
Here, we highlight the main idea in the decomposition. We first have the following two facts: 
\begin{itemize}
    \item If $\theta=0$, then the $\nu^{1/3}$ size for vorticity is enough for the asymptotic stability \cite{masmoudizhao2022, weizhang2023nonlinear}, which also seems to be optimal. 
    \item The temperature $\theta$ has enhanced dissipation. Without the derivative loss in the buoyancy term $\partial_x\theta$, a natural size requirement of temperature is $\nu^{2/3}$. The additional $\nu^{1/6}$ size loss in \cite{MasmoudiZhaiZhao2022, zhang2023stability} is used to propagate the additional $1/3$ derivative in the $x$ direction. 
\end{itemize}
We first introduce the equation of $\theta^{\mathrm{i}}$ which carries the information of the initial temperature and propagates more regularity in the $x$ direction. It is natural to use the linearized equation. Here, we modify the equation by adding the interaction with the zero mode, namely, $u^{\mathrm{i}, 1}_0\partial_x\theta^{\mathrm{i}}$, since it decays slower than the non-zero modes but can propagate additional derivatives in the $x$ direction. Then we discuss the equation of $(\theta^{\mathrm{e}}, \omega^{\mathrm{e}})$ \eqref{eq:e} which has zero initial data. The forcing term $u_{\neq}^{\mathrm{i}}\cdot \nabla \theta^{\mathrm{i}}$ will determine the size of $(\theta^{\mathrm{e}}, \omega^{\mathrm{e}})$. An easy calculation (see Remark \ref{force}) gives that the forcing term 
$\|\Lambda_t^n(u_{\neq}^{\mathrm{i}}\cdot\nabla \th^{\mathrm{i}})\|_{L^2}\lesssim \frac{\ep_0^2\nu}{t^2+1}$ is good (with enough decay rate and smallness).

We then expect a smaller size of $(\theta^{\mathrm{e}}, \omega^{\mathrm{e}})\lesssim (\nu, \nu^{\f23})$ than that of $(\theta^{\mathrm{i}}, \omega^{\mathrm{i}})\lesssim (\nu^{\f23},\nu^{\f13})$ formally. However, to ensure the linear forcing term with derivative loss $\partial_x\theta^{\mathrm{e}}$ and the quasi-linear terms $u^{\mathrm{i}}\cdot\nabla(\omega^{\mathrm{e}},\theta^{\mathrm{e}})$ and $u^{\mathrm{e}}\cdot\nabla(\omega^{\mathrm{i}}, \theta^{\mathrm{i}})$ do not increase the size, we have to treat them carefully. A classical discussion in \cite{zhang2023stability, MasmoudiZhaiZhao2022} gives us that a workable size assumption for this system is 
\begin{align*}
    (\theta^{\mathrm{e}}, |D_x|^{\f13}\theta^{\mathrm{e}}, \omega^{\mathrm{e}})\lesssim (\nu^{\beta+\f12},\nu^{\beta+\f13}, \nu^{\beta}),
\end{align*}
for some $1/2>\beta\geq \f13$. The main argument is that the propagation of additional $1/3$ derivative in the $x$ direction breaks the transport structure, for example $u^{\mathrm{i},2}\partial_y \theta^{\mathrm{e}}$, which requires additional $\nu^{\f16}$ size assumption due to the natural balance from the diffusion or enhanced dissipation, namely, $\nu k^2$, where $k$ is the wave number with respect to the $x$ variable. 
The other problematic term is $u^{\mathrm{e},2}\partial_y\theta^{\mathrm{i}}$, which could be treated as the reaction term. Our observation is for short-time region $t\leq \nu^{-\f16}$, such a balance between $k$ and $\nu$ is weak. Indeed, we can even propagate one additional derivative in the $x$ direction with only causing $\nu^{-\f16}$ loss, which allows us the show that $\partial_x\theta^{\mathrm{e}}\lesssim \nu^{1-\f16}$. Then the linear forcing term $\partial_x\theta^{\mathrm{e}}$ in the equation of $\omega^{\mathrm{e}}$ only cause a total $\nu^{-\f16-\f16}$ loss for the time region $[0,\nu^{-\f16}]$. Thus we have
\begin{align*}
    (\theta^{\mathrm{e}}, \partial_x\theta^{\mathrm{e}}, \omega^{\mathrm{e}})\lesssim (\nu,\nu^{\f56}, \nu^{\f23}),\quad t\in [0,\nu^{-\f16}]. 
\end{align*}
 We observe that the transport structure is less important if the velocity field has a size less than $\nu^{\f12}$. For $t\geq \nu^{-\f16}$, we use inviscid damping and dissipation, namely, 
 \begin{align*}
     \|\Lambda_t^n(u^{\mathrm{i},2}\partial_y \theta^{\mathrm{e}})\|_{L^1L^2}
     &\lesssim \|\Lambda_t^nu^{\mathrm{i},2}\|_{L^2L^2}
     \|\partial_y \Lambda_t^n \theta^{\mathrm{e}}\|_{L^2L^2}\\
     &\lesssim \left\|\f{1}{1+t}\Lambda_t^{n+1}u_{\neq}^{\mathrm{i},1}\right\|_{L^2L^2}\|\partial_y \Lambda_t^n \theta^{\mathrm{e}}\|_{L^2L^2}\\
     &\lesssim \nu^{\f16}\left\|\Lambda_t^{n+1}u_{\neq}^{\mathrm{i},1}\right\|_{L^2L^2}\|\partial_y \Lambda_t^n \theta^{\mathrm{e}}\|_{L^2L^2}.
 \end{align*}
 We only need to treat the same problematic term  $u^{\mathrm{e},2}\partial_y\theta^{\mathrm{i}}$, which behaves as $\langle t\rangle u^{\mathrm{e},2}\theta^{\mathrm{i}}$. We assume the size of $\omega^{\mathrm{e}}$ is $\nu^{\beta}$. We consider the temperature equation. On one hand, formally speaking, this term is of size $\nu^{\beta+\f23}$. On the other hand, this term has a similar behavior as the reaction term, which could lose $\nu^{\f13}$ as $t$ is large. Thus, we expect the size of $\theta^{\mathrm{e}}$ is smaller than $\nu^{\beta+\f13}$. For the linear forcing term, which requires that $|D_x|^{\f13}\theta^{\mathrm{e}}$ is of a size smaller than $\nu^{\beta+\f13}$. Therefore we use a different size for $t\geq \nu^{-\f16}$, namely, 
\begin{align*}
    (\theta^{\mathrm{e}}, |D_x|^{\f13}\theta^{\mathrm{e}}, \omega^{\mathrm{e}})\lesssim (\nu^{\f56}, \nu^{\f56}, \nu^{\f12}),\quad t\in [\nu^{-\f16},\infty).
\end{align*}

At last, we want to point out that to propagate additional $1/3$ derivative, we study the simplified equation $\partial_t\theta^{\mathrm{e}}=u^{\mathrm{e},2}\partial_y\theta^{\mathrm{i}}$ and derive a toy model in the critical time interval $t\in [\f{2\eta}{2k+1}, \f{2\eta}{2k-1}]$:
\begin{align*}
    \partial_t f_1\approx \frac{|k|^{\f13}f_2 }{k^2+(\eta-kt)^2}(\nu^{\f23}te^{-c\nu^{\f13}t}). 
\end{align*}
Here $f_1$ represents the temperature $|\partial_x|^{\f13}\theta^{\mathrm{e}}$, $f_2$ represents the vorticity $u^{\mathrm{e},2}$, and $\partial_y\theta^{\mathrm{i}}$ is replaced by $\nu^{\f23}te^{-c\nu^{\f13}t}\lesssim \nu^{\f13}$. We expect that $\nu^{\f13}f_2$ has the same size as $f_1$, which allows us to study 
\begin{align*}
    \partial_t f_1\approx \frac{|k|^{\f13}f_1 }{k^2+(\eta-kt)^2}, 
\end{align*}
in each critical time interval $t\in [\f{2\eta}{2k+1}, \f{2\eta}{2k-1}]$. This gives us a growth $e^{\pi\frac{|k|^{\f13}}{k^2}}$ for each critical time interval and leads us to a total growth
\begin{align*}
    \prod_{k=1}^{\sqrt{|\eta|}}e^{\pi\frac{|k|^{\f13}}{k^2}}\lesssim e^{3\pi}. 
\end{align*}
In terms of the modification of the Fourier multiplier, we modified the Fourier multiplier $\mathcal{M}_3$ by adding $\mu$ derivatives in the $x$ direction with $2/3\leq \mu<1$.

\section{Multipliers and key lemmas}
\subsection{Construction of the multipliers}  In this section, we introduce the Fourier multipliers and their main properties. We choose a real-valued function $\varphi\in C^{\infty}(\mathbb R)$ such that $0\leq\varphi\leq1$, $0\leq \varphi'\leq\f14$ on $\mathbb R$, and $\varphi'=\f14$ on $[-1,1]$. Let 
\begin{align*}
&\mathcal{M}_1(k, \xi)=\left\{\begin{aligned}
\varphi\big(\nu^{\f13}k^{-\f13}\mathrm{sgn}(k)\xi\big),\quad &k\neq 0\\
0,\quad &k=0,
\end{aligned}\right.\\
&\mathcal{M}_2( t,k, \xi)=\left\{\begin{aligned}
\f{\pi}{2}+\arctan\big(\f{\xi}{k}\big),\quad &0<|k|\leq\nu^{-\f12},\\
0,\quad &k=0,\ \mathrm{or}\  |k|>\nu^{-\f12},
\end{aligned}\right.\\
&\mathcal{M}_3(t, k, \xi)=
\left\{\begin{aligned}\sum_{l\in \mathbb{Z}\setminus\{0\}}\f{|l|^{\mu}}{|l|^2}\Big(\mathrm{sgn}(l)\arctan\Big(\f{\xi+t(k-l)}{1+|k-l|+|l|}\Big)+\f{\pi}{2}\Big),\quad &|k|\leq \nu^{-\f12},\\
0,\quad &|k|>\nu^{-\f12},
\end{aligned}\right.
\end{align*} 
and
\begin{align*}
\mathcal{M}=e^{\d_0\nu^{\f13}t}(\mathcal{M}_1+\mathcal{M}_2+\mathcal{M}_3+1).
\end{align*}
$\mathcal{M}_t=\mathcal{M}(t, D_x,D_y)$ is a self-adjoint Fourier multiplier. Here $\mathcal M_2$ is used to capture the inviscid damping, $\mathcal{M}_1$ together with the diffusion gives us the enhanced dissipation, and the operator $\mathcal{M}_3$ was first introduced in \cite{weizhang2023nonlinear} to capture the nonlinear growth of the reaction. We modify $\mathcal{M}_3$ by adding $\mu<1$ derivatives in the $x$ direction. 
And the following estimates hold
\begin{align*}
1\leq \mathcal{M}_1+\mathcal{M}_2+\mathcal{M}_3+1\leq  C_{\mu},\quad 1\leq \mathcal{M}\leq C_{\mu}e^{\d_0\nu^{\f13}t},\quad  e^{\d_0\nu^{\f13}t}  \mathcal{M}|_{k\neq0}.
\end{align*}
We also define the multiplier $\Upsilon_t=\Upsilon(t,D_x,D_y)$, with
\begin{align*}
    \Upsilon(t,k,\xi)=\sum_{l\in\mathbb Z\backslash\{0\}}\frac{|l|^{\mu}}{|l|}\frac{1+|k-l|+|l|}{(1+|k-l|+|l|)^2+(\xi+t(k-l))^2}> 0,
\end{align*}
where $2/3\leq\mu<1$, it holds that
\begin{equation}\label{Upsilon}
    (-\pa_t+k\pa_{\xi})\mathcal{M}_3(t,k,\xi)=\Upsilon(t,k,\xi)\quad |k|\leq \nu^{-\f12}.
\end{equation}

\subsection{Basic properties of the multipliers}
Taking inner product of $(\pa_t+y\pa_x-\nu\Delta)\omega$ with $\mathcal{M}_t\omega$, we obtain
\begin{align*}
    2\mathrm{Re}\langle\pa_t\omega,\mathcal{M}_t\omega\rangle+2\mathrm{Re}\langle y\pa_x\omega,\mathcal{M}_t\omega\rangle-2\nu\mathrm{Re}\langle\Delta\omega,\mathcal{M}_t\omega\rangle\overset{def}{=}R.
\end{align*}
Since $\mathcal{M}_t$ is self-adjoint and $y\pa_x$ is skew-adjoint, we have
\begin{align*}
    2\mathrm{Re}\langle y\pa_x\omega,\mathcal{M}_t\omega\rangle=\langle [\mathcal{M}_t,y\pa_x]\omega,\omega\rangle=\langle k\pa_{\xi}\mathcal{M}(t,D)\omega,\omega\rangle,\\
    2\mathrm{Re}\langle \pa_t\omega,\mathcal{M}_t\omega\rangle=\pa_t\langle \omega,\mathcal{M}_t\omega\rangle-\langle \omega,\pa_{t}\mathcal{M}(t,D)\omega\rangle.
\end{align*}
Then, by Plancherel's formula, we have
\begin{align*}
    \frac{d}{dt}\langle \omega,\mathcal{M}_t\omega\rangle+\sum_{k}\int_{\mathbb R}\big((k\pa_{\xi}-\pa_t)\mathcal{M}+2\nu(k^2+\xi^2)\mathcal{M}\big)(t,k,\xi)|\hat{\omega}(t,k,\xi)|^2d\xi=R.
\end{align*}
Next, we give a lower bound of $(k\pa_{\xi}-\pa_t)\mathcal{M}+2\nu(k^2+\xi^2)\mathcal{M}$ in the following lemma. 
\begin{lemma}[\cite{weizhang2023nonlinear}]
It holds that
    \begin{equation}\label{energy-dissipation}
\begin{aligned}
    &2\nu\left\|(D_x,D_y)\sqrt{\mathcal{M}_t}f\right\|_{L^2}^2+\langle (k\pa_{\xi}-\pa_t)\mathcal{M}(t,D)f,f\rangle\\
    \geq&\frac{\delta_0}{2}\Big(\nu\left\|(D_x,D_y)\sqrt{\mathcal{M}_t}f\right\|_{L^2}^2+\nu^{\f13}\left\||D_x|^{\frac{1}{3}}\sqrt{\mathcal{M}_t}f\right\|_{L^2}^2+\left\|D_x\sqrt{\mathcal{M}_t}(-\Delta)^{-\f12}f\right\|_{L^2}^2\\
    &+\left\|\sqrt{\mathcal{M}_t\Upsilon_t}f\right\|_{L^2}^2\Big).
\end{aligned}
\end{equation}
\end{lemma}
\begin{proof}
According to the definition of $\mathcal{M}_1$, we have
\begin{equation}\label{kM1}
\begin{aligned}
    k\pa_{\xi}\mathcal{M}_1(t,k,\xi)=\nu^{\f13}|k|^{\f23}\varphi'\big(\nu^{\f13}|k|^{-\f13}\mathrm{sgn}(k)\xi\big),\\
    \nu(k^2+\xi^2)+k\pa_{\xi}\mathcal{M}_1(t,k,\xi)\geq\f14\nu^{\f13}|k|^{\f23},\quad \forall \xi\in\mathbb R,k\in\mathbb Z.
\end{aligned}
\end{equation}
We also have
\begin{equation}\label{kM2}
    k\pa_{\xi}\mathcal{M}_2(t,k,\xi)=\frac{k^2}{k^2+\xi^2},\quad |k|\leq\nu^{-\f12}.
\end{equation}
Then combined with \eqref{Upsilon} and $\pa_t\mathcal{M}_1=\pa_t\mathcal{M}_2=0$, we infer that for $|k|\leq\nu^{-\f12}$,
\begin{align*}
    &2\nu(k^2+\xi^2)\big(\mathcal{M}_1+\mathcal{M}_2+\mathcal{M}_3+1\big)+(k\pa_{\xi}-\pa_t)\big(\mathcal{M}_1+\mathcal{M}_2+\mathcal{M}_3+1\big)\\
    \geq &\nu(k^2+\xi^2)+\frac{k^2}{k^2+\xi^2}+\f14\nu^{\f13}|k|^{\f23}+\Upsilon(t,k,\xi).
\end{align*}
If $|k|>\nu^{-\f12}$, then $\mathcal{M}_2=\mathcal{M}_3=0$, and 
\begin{align*}
    2\nu(k^2+\xi^2)\big(\mathcal{M}_1+\mathcal{M}_2+\mathcal{M}_3+1\big)+(k\pa_{\xi}-\pa_t)\big(\mathcal{M}_1+\mathcal{M}_2+\mathcal{M}_3+1\big)\geq \nu(k^2+\xi^2)+\f14\nu^{\f13}|k|^{\f23}.
\end{align*}
Notice that 
\begin{align*}
    0\leq\Upsilon(t,k,\xi)\leq\sum_{l\in\mathbb Z\backslash\{0\}}\frac{|l|^{\mu}}{|l|}\frac{1}{(1+|k-l|+|l|)}\leq \sum_{l\in\mathbb Z\backslash\{0\}}\frac{|l|^{\mu}}{|l|}\frac{1}{(1+|l|)}\lesssim_{\mu}1.
\end{align*}
Then for $|k|>\nu^{-\f12}$, we have $k^2/(k^2+\xi^2)+\Upsilon(t,k,\xi)\leq C_{1,\mu}\nu^{\f13}|k|^{\f23}$ and 
\begin{align*}
    &2\nu(k^2+\xi^2)\big(\mathcal{M}_1+\mathcal{M}_2+\mathcal{M}_3+1\big)+(k\pa_{\xi}-\pa_t)\big(\mathcal{M}_1+\mathcal{M}_2+\mathcal{M}_3+1\big)\\
    \geq&\nu(k^2+\xi^2)+\frac{1}{8}\nu^{\f13}|k|^{\f23}+\frac{k^2}{8C_{1,\mu}(k^2+\xi^2)}+\frac{\Upsilon(t,k,\xi)}{8C_{1,\mu}},
\end{align*}
compared with $|k|\leq\nu^{-\f12}$ case, this result still holds for $|k|\leq\nu^{-\f12}$. Then we have
\begin{align*}
    2\nu(k^2+\xi^2)\mathcal{M}+(k\pa_{\xi}-\pa_t)\mathcal{M}\geq& e^{\delta_0\nu^{\f13}t}\big(\nu(\xi^2+k^2)+\f18\nu^{\f13}|k|^{\f23}\\
    &+\frac{k^2}{8C_{1,\mu}(k^2+\xi^2)}+\frac{\Upsilon(t,k,\xi)}{8C_{1,\mu}}-\delta_0\nu^{\f13}|k|^{\f23}(\mathcal{M}_1+\mathcal{M}_2+\mathcal{M}_3+1)\big).
\end{align*}
Since $\mathcal{M}_1+\mathcal{M}_2+\mathcal{M}_3+1\leq C_{\mu}$, by taking $\delta_0=\min\{\f{1}{12C_{\mu}},\f{1}{4C_{1,\mu}}\}$, we have $\f18-\delta_0(\mathcal{M}_1+\mathcal{M}_2+\mathcal{M}_3+1)\geq\f{1}{24}$, and
\begin{align*}
    2\nu(k^2+\xi^2)\mathcal{M}+(k\pa_{\xi}-\pa_t)\mathcal{M}\geq&\f{e^{\delta_0\nu^{\f13}t}}{8C_{1,\mu}}\Big(\nu(\xi^2+k^2)+\nu^{\f13}|k|^{\f23}
    +\frac{k^2}{k^2+\xi^2}+\Upsilon(t,k,\xi)\Big)\\
    \geq& \f{\delta_0\mathcal{M}}{2}\Big(\nu(\xi^2+k^2)+\nu^{\f13}|k|^{\f23}
    +\frac{k^2}{k^2+\xi^2}+\Upsilon(t,k,\xi)\Big).
\end{align*}
Then the result follows directly.
\end{proof}
Next, we show the estimates of the derivative of $\mathcal{M}$, which are used in commutator estimates. 
\begin{lemma}[\cite{weizhang2023nonlinear}]\label{de-M}
    It holds for $k\neq0$ that
    \begin{align*}
        |\pa_{\xi}\mathcal{M}|\leq C\mathcal{M}(\nu^{\f13}|k|^{-\f13}+|k|^{-1}),
    \end{align*}
    and for $k>\nu^{-\f12}$, it holds that
    \begin{align*}
         |\pa_k\mathcal{M}|\lesssim \mathcal{M}\nu^{\f13}|k|^{-\f43}|\xi|\lesssim \mathcal{M}\nu^{\f12}|k|^{-1}|\xi|.
    \end{align*}
\end{lemma}
\begin{proof}
By \eqref{kM1} and \eqref{kM2}, we have
\begin{align*}
    |\pa_{\xi}\mathcal{M}_1|\leq \nu^{\f13}|k|^{-\f13},\quad |\pa_{\xi}\mathcal{M}_2|\leq |k|^{-1}, \quad k\neq0.
\end{align*}
By the definition of $\mathcal{M}_3$, we have
\begin{align*}
    |\pa_{\xi}\mathcal{M}_3|\leq \sum_{l\in\mathbb Z\backslash\{0\}}\frac{|l|^{\mu}}{|l|^2}\frac{1}{1+|k-l|+|l|}\leq \sum_{l\in\mathbb Z\backslash\{0\}}\frac{|l|^{\mu}}{|l|^2}\frac{1}{1+|k|}\leq\frac{C}{1+|k|}.
\end{align*}
Thus, for $k\neq0$, we have
\begin{align*}
    |\pa_{\xi}(\mathcal{M}_1+\mathcal{M}_2+\mathcal{M}_3+1)|\leq C(\nu^{\f13}|k|^{-\f13}+|k|^{-1}),\quad \pa_{\xi}\mathcal{M}\leq C\mathcal{M}(\nu^{\f13}|k|^{-\f13}+|k|^{-1}).
\end{align*}
For $|k|>\nu^{-\f12}$, we have $\mathcal{M}=e^{\delta_0\nu^{\f13}t}(\mathcal{M}_1+1)$, and 
\begin{align*}
    |\pa_k\mathcal{M}_1(t,k,\xi)|=\f13\nu^{\f13}|k|^{-\f43}|\xi|\varphi'(\nu^{\f13}|k|^{-\f13}\mathrm{sgn}(k)\xi)\leq \nu^{\f13}|k|^{-\f43}|\xi|,
\end{align*}
which gives
\begin{align*}
    |\pa_k\mathcal{M}|\leq \mathcal{M}\nu^{\f13}|k|^{-\f43}|\xi|\leq \mathcal{M}\nu^{\f12}|k|^{-1}|\xi|,\quad |k|>\nu^{-\f12}.
\end{align*}
\end{proof}
We also lay out the following estimates, which will be used in later calculations.

For $t\geq 0$ and $b\in\mathbb R$, we know $\Lambda_t^b(k,\xi)\leq\sqrt{\mathcal{M}(t,k,\xi)}\Lambda_t^b(k,\xi)$, which implies
\begin{equation}\label{eq:M}
    \|\Lambda_t^bf\|_{L^2}\leq\|\sqrt{\mathcal{M}_t}\Lambda_t^bf\|_{L^2}.
\end{equation}
For $k\neq0$ and $t\geq0$, we have $e^{\delta_0/2\nu^{\f13}t}\Lambda_t^b(k,\xi)\leq \sqrt{\mathcal{M}(t,k,\xi)}\Lambda_t^b(k,\xi)$, which gives
\begin{equation}\label{eq:Mneq}
    \|e^{\f{\delta_0}{2}\nu^{\f13}t}\Lambda_t^bf_{\neq}\|_{L^2}\leq \|\sqrt{\mathcal{M}_t}\Lambda_t^bf_{\neq}\|_{L^2}.
\end{equation}
For $k\neq0$ and $t>0$, we have
\begin{align*}
    |\xi|\leq|\xi+tk|+t|k|\leq(1+t)\Lambda_t(k,\xi)\leq C\nu^{-\f13}e^{\f{\delta_0}{2}\nu^{\f13}t}\Lambda_t(k,\xi)\leq C\nu^{-\f13}\sqrt{\mathcal{M}(t,k,\xi)}\Lambda_t(k,\xi).
\end{align*}
For $k=0$ and $t>0$, we have
\begin{align*}
    |\xi|=|\xi+tk|\leq\Lambda_t(k,\xi)\leq \sqrt{\mathcal{M}(t,k,\xi)}\Lambda_t(k,\xi).
\end{align*}
Then for any $d\geq 0$, we have
\begin{equation}\label{eq:Mf}
    \|D_y\Lambda_t^{d}f\|_{L^2}\leq C\nu^{-\f13}\|\sqrt{\mathcal{M}_t}\Lambda_t^{d+1}f\|_{L^2}.
\end{equation}
\begin{lemma}[\cite{weizhang2023nonlinear}]\label{Lem:L1}
    If $P_0f=0$, then for $t\geq0$, we have
    \begin{align*}
        \|\hat{f}\|_{L^1}\leq C(1+t)^{-1}\|\nabla\Lambda_tf\|_{L^2}.
    \end{align*}
\end{lemma}
\begin{proof}
    As $P_0f=0$, by H\"older's inequality, we have
    \begin{align*}
        &\|\hat{f}\|_{L^1}=\sum_{k\in\mathbb Z\backslash\{0\}}\int_{\mathbb R}|\hat{f}(k,\xi)|d\xi\\
        \leq& \Big(\sum_{k\in\mathbb Z\backslash\{0\}}\int_{\mathbb R}(k^2+\xi^2)\Lambda_t^2(k,\xi)|\hat{f}(k,\xi)|^2d\xi\Big)^{\f12}\Big(\sum_{k\in\mathbb Z\backslash\{0\}}\int_{\mathbb R}\f{d\xi}{(k^2+\xi^2)\Lambda_t^2(k,\xi)}\Big)^{\f12}=:K_1^{\f12}K_2^{\f12}.
    \end{align*}
    By Plancherel's formula, we know $K_1=\|\nabla\Lambda_tf\|_{L^2}^2$. Next for $K_2$, we use the fact that for $a>0$, $s>0$, and $z\in\mathbb R$:
\begin{equation}\label{eq:as}
    \int_{\mathbb R}\f{d\eta}{(a^2+\eta^2)(s^2+(z-\eta)^2)}=\frac{\pi}{as}\f{a+s}{(a+s)^2+z^2}.
\end{equation}
Then by taking $a=s=|k|, z=-tk$ in \eqref{eq:as}, we have
    \begin{align*}
        K_2=\sum_{k\in\mathbb Z\backslash\{0\}}\int_{\mathbb R}\f{d\xi}{(k^2+\xi^2)(1+k^2+(\xi+tk)^2)}\leq \sum_{k\in\mathbb Z\backslash\{0\}}\frac{\pi}{k^2}\f{2|k|}{(2k)^2+(tk)^2}\leq \f{C}{(1+t)^2},
    \end{align*}
    which impiles $\|\hat{f}\|_{L^1}\leq K_1^{\f12}K_2^{\f12}\leq C(1+t)^{-1}\|\nabla\Lambda_tf\|_{L^2}$.
\end{proof}
The following lemma gives the inviscid damping estimate. 
\begin{lemma}[\cite{weizhang2023nonlinear}]\label{Lem: ID}
    Let $b\geq2$, $t\geq0$ and $b_1>0$. Let $(u^1, u^2)=\nabla^{\bot}(-\Delta)^{-1}\omega$. Then
    \begin{align*}
        \|\Lambda_t^{b_1}u_{\neq}^1\|_{L^2}+(1+t)\|\Lambda_t^{b_1}u^2\|_{L^2}\leq C(1+t)^{-1}\|\Lambda_t^{b+b_1}\omega_{\neq}\|_{L^2}.
    \end{align*}
\end{lemma}
\begin{proof}
    For $k\in\mathbb Z\backslash\{0\}$ and $\xi\in\mathbb R$, we have
    \begin{align*}
        (k^2+\xi^2)\Lambda_t^2(t,\xi)\geq&(k^2+\xi^2)(k^2+(\xi+kt)^2)=k^4\big(1+(\frac{\xi}{k})^2\big)\big(1+(\frac{\xi}{k}+t)^2\big)\\
        \geq&C^{-1}k^4(1+t)^2\geq C^{-1}|k|(1+t)^2\geq C^{-1}(1+t)^2.
    \end{align*}
    Since $u_{\neq}^1=-\pa_y(-\Delta)^{-1}\omega_{\neq}$, by Plancherel's formula, we have
    \begin{align*}
        \|\Lambda_t^{b_1}u_{\neq}^1\|_{L^2}=&\|\Lambda_t^{b_1}(t,\xi)\xi(k^2+\xi^2)^{-1}\hat{\omega}_{\neq}\|_{L^2}
        \leq\|\Lambda_t^{b_1}(t,\xi)(k^2+\xi^2)^{-\frac{1}{2}}\hat{\omega}_{\neq}\|_{L^2}\\
        \leq&C(1+t)^{-1}\|\Lambda_t^{b_1}(t,\xi)\Lambda_t(t,\xi)\hat{\omega}_{\neq}\|_{L^2}\\
        \leq&C(1+t)^{-1}\|\Lambda_t^{b_1+b}(t,\xi)\hat{\omega}_{\neq}\|_{L^2}.
    \end{align*}
    Similarly, from $u^2=u_{\neq}^2=\pa_x(-\Delta)^{-1}\omega_{\neq}$ and Plancherel's formula, we have
    \begin{align*}
        \|\Lambda_t^{b_1}u^2\|_{L^2}=&\|\Lambda_t^{b_1}(t,\xi)k(k^2+\xi^2)^{-1}\hat{\omega}_{\neq}\|_{L^2}\leq C(1+t)^{-2}\|\Lambda_t^{b_1+b}(t,\xi)\hat{\omega}_{\neq}\|_{L^2}.
    \end{align*}
\end{proof}
\subsection{Nonlinear interaction}
In this section, we study the multipliers acting on the nonlinear terms. The first lemma is a commutator estimate which will be frequently used in the short-time region.
\begin{lemma}\label{Lem: transport-shorttime}
Let $v=(v^1, v^2)=\nabla^{\bot}(-\Delta)^{-1}w$. 
    Then, for $d\geq b>2$ it holds that 
    \begin{align*}
        \left|\langle \Lambda_t^d(v\cdot \nabla f), \Lambda_t^df\rangle\right|
        \lesssim (1+t)\|\Lambda_t^bw\|_{L^2}\|\Lambda_t^df\|_{L^2}^2
        +(1+t)\|\Lambda_t^dw\|_{L^2}\|\Lambda_t^df\|_{L^2}\|\Lambda_t^{b}f\|_{L^2},
    \end{align*} 
    in particular, if $v=(v_0^1,0)$, it holds that
    \begin{align*}
        \left|\langle \Lambda_t^d(v\cdot \nabla f), \Lambda_t^df\rangle\right|
        \lesssim \|\Lambda_t^bw\|_{L^2}\|\Lambda_t^df\|_{L^2}^2
        +\|\Lambda_t^dw\|_{L^2}\|\Lambda_t^df\|_{L^2}\|\Lambda_t^{b}f\|_{L^2}.
    \end{align*}
\end{lemma}
\begin{proof}
    Since $\dive\, v=0$, we know
\begin{align*}
    \langle v\cdot\nabla(\Lambda_t^df),\Lambda_t^df\rangle=0.
\end{align*}
Then by Plancherel's formula, we have
\begin{align*}
    &\langle \Lambda_t^d\big(v\cdot\nabla f\big),\Lambda_t^d f\rangle=\langle \Lambda_t^d\big(v\cdot\nabla f\big)-v\cdot\nabla(\Lambda_t^df),\Lambda_t^df\rangle\\
    =&\sum_{k,l}\iint\big(\Lambda_t^d(k,\xi)-\Lambda_t^d(k-l,\xi-\eta)\big)\hat{v}(l,\eta)\cdot i(k-l,\xi-\eta)\\
    &\times\hat{f}(k-l,\xi-\eta)\cdot\Lambda_t^d(k,\xi)\overline{\hat{f}(k,\xi)}d\xi d\eta.
\end{align*}
By Taylor's formula, for some $\alpha\in[0,1]$, we have
\begin{align*}
    &|\Lambda_t^d(k,\xi)-\Lambda_t^d(k-l,\xi-\eta)|\\
    =&\frac{d}{2}\left|\Lambda_t^2(k,\xi)-\Lambda_t^2(k-l,\xi-\eta)\right|\left|\alpha\Lambda_t^2(k,\xi)+(1-\alpha)\Lambda_t^2(k-l,\xi-\eta)\right|^{\frac{d}{2}-1},
\end{align*}
and since
\begin{align*}
    |\Lambda_t^2(k,\xi)-\Lambda_t^2(k-l,\xi-\eta)|=&|k^2+(\xi+kt)^2-(k-l)^2-(\xi-\eta+t(k-l))^2|\\
    =&|l(2k-l)+(\eta+lt)(2\xi+2kt-\eta-lt)|\\
    \leq&C\big(|l|+|\eta+lt|\big)\left|\Lambda_t(l,\eta)+\Lambda_t(k-l,\xi-\eta)\right|.
\end{align*}
We know
\begin{align*}
&|\Lambda_t^d(k,\xi)-\Lambda_t^d(k-l,\xi-\eta)|\\
    \leq&C|l(2k-l)+(\eta+lt)(2\xi+2kt-\eta-lt)|\left|\Lambda_t^{d-2}(l,\eta)+\Lambda_t^{d-2}(k-l,\xi-\eta)\right|\\
    \leq&C\big(|l|+|\eta+lt|\big)\left|\Lambda_t^{d-1}(l,\eta)+\Lambda_t^{d-1}(k-l,\xi-\eta)\right|.
\end{align*}
Then, we have
\begin{align*}
    \left|\langle \Lambda_t^d(v\cdot \nabla f), \Lambda_t^df\rangle\right|\leq &C\sum_{k,l}\iint\big(|l|+|\eta+lt|\big)\left|\Lambda_t^{d-1}(l,\eta)+\Lambda_t^{d-1}(k-l,\xi-\eta)\right|\\
    &\times |\hat{v}(l,\eta)|\left|(k-l,\xi-\eta)\hat{f}(k-l,\xi-\eta)\right|\cdot\Lambda_t^d(k,\xi)\left|\overline{\hat{f}(k,\xi)}\right|d\xi d\eta.
\end{align*}
We know the fact that
\begin{align*}
        |l|+|\eta+lt|\leq C\Lambda_t(l,\eta)\leq C\frac{(l^2+\eta^2)^{\f12}}{(l^2+\eta^2+1)^{\f12}}\Lambda_t(l,\eta);
    \end{align*}
 and $v=\nabla^{\perp}(-\Delta)^{-1}w$, we know $(l^2+\eta^2)^{\f12}|\hat{v}(l,\eta)|=|\hat{w}(l,\eta)|$ and 
    \begin{align*}
        \big(|l|+|\eta+lt|\big)|\hat{v}(l,\eta)|\leq C\frac{(l^2+\eta^2)^{\f12}}{(l^2+\eta^2+1)^{\f12}}\Lambda_t(l,\eta)|\hat{v}(l,\eta)|\leq C\frac{\Lambda_t(l,\eta)}{(l^2+\eta^2+1)^{\f12}}|\hat{w}(l,\eta)|.
    \end{align*}   
Moreover, for $(k',\xi')=(k-l,\xi-\eta)$, we have
    \begin{align*}
        |(k',\xi')|\leq|\xi'+tk'|+(1+t)|k'|\leq(1+t)\Lambda_t(k',\xi').
    \end{align*}
    Then by Young's convolution inequality, we get
\begin{align*}
    \left|\langle \Lambda_t^d(v\cdot \nabla f), \Lambda_t^df\rangle\right|\leq&C\sum_{k,l}\iint \frac{\Lambda_t(l,\eta)}{(l^2+\eta^2+1)^{\f12}}|\hat{w}(l,\eta)|\left|\Lambda_t^{d-1}(l,\eta)+\Lambda_t^{d-1}(k-l,\xi-\eta)\right|\\
    &\times (1+t)\left|\Lambda_t(k-l,\xi-\eta)\hat{f}(k-l,\xi-\eta)\right|\cdot\Lambda_t^d(k,\xi)\left|\overline{\hat{f}(k,\xi)}\right|d\xi d\eta\\
    \leq&C(1+t)\Big(\|\Lambda^{-1}\Lambda_t\hat{w}\|_{L^1}\|\Lambda_t^df\|_{L^2}\|\Lambda_t^df\|_{L^2}+\|\Lambda^{-1}\Lambda_t^d\hat{w}\|_{L^{\f43}}\|\Lambda_t\hat{f}\|_{\f43}\|\Lambda_t^df\|_{L^2}\Big),
\end{align*}
where $\Lambda(l,\eta)=(l^2+\eta^2+1)^{\f12}$. Next, since $d-1\geq1$ and $\|\Lambda^{-1}\|_{L^4}\leq C\|\Lambda_t^{-1}\|_{L^4}\leq+\infty$, we have
    \begin{align*}
        &\|\Lambda^{-1}\widehat{\Lambda_tw}\|_{L^1}\leq C\|\widehat{\Lambda_tw}\|_{L^{\f43}}\leq C\|\Lambda_t^2w\|_{L^2}\leq C\|\Lambda_t^bw\|_{L^2},\\
        &\|\Lambda^{-1}\widehat{\Lambda_t^dw}\|_{L^{\f43}}\leq\|\Lambda_t^dw\|_{L^2},\\
        &\|\Lambda_t\hat{f}\|_{\f43}\leq C\|\Lambda^{-1}\Lambda_t^2\hat{f}\|_{L^{\f43}}\leq C\|\Lambda_t^bf\|_{L^2}.
    \end{align*}
For the zero mode case, we can get the result by taking $l=0$ in the proof. The lemma then follows directly. 
\end{proof}
The next lemma is used to estimate reaction terms.
\begin{lemma}\label{Lem:I11}
    Let $v=(v^1,v^2)=\nabla^{\perp}(-\Delta)^{-1}w$, then for $b>2$, it holds for $t\geq \nu^{-\f16}$ that
    \begin{align*}
        &\left|\big\langle \sqrt{\mathcal{M}_t}\Lambda_t^b(v^2\pa_yf),\sqrt{\mathcal{M}_t}\Lambda_t^bg\big\rangle\right| \\
        \leq&C\|\nabla\sqrt{\mathcal{M}_t}\Lambda_t^bv^2\|_{L^2}\big(\nu^{-\f13}\|\sqrt{\mathcal{M}_t}\Lambda_t^bf\|_{L^2}\|\sqrt{\Upsilon_t\mathcal{M}_t}\Lambda_t^bg\|_{L^2}+\nu^{\f16}\|D_y\Lambda_t^bf\|_{L^2}\|\sqrt{\mathcal{M}_t}\Lambda_t^bg\|_{L^2}\big) .
    \end{align*}
\end{lemma}
\begin{proof}
    By Plancherel's formula, we have
         \begin{align*}
	\left|\big\langle \sqrt{\mathcal{M}_t}\Lambda_t^b(v^2\pa_yf),\sqrt{\mathcal{M}_t}\Lambda_t^bg\big\rangle\right|
	&=\bigg|\sum_{k,l}\iint\sqrt{\mathcal{M}(t,k,\xi)}\Lambda_t^b(k,\xi)\hat{v}^{2}(l,\eta)\\
	&\times i(\xi-\eta)\hat{f}(k-l,\xi-\eta)\sqrt{\mathcal{M}(t,k,\xi)}\Lambda_t^b(k,\xi)\overline{\hat{g}(k,\xi)}d\xi d\eta\bigg|=:\mathcal{I}.
	\end{align*}
	Combining with the bounds of $\mathcal{M}$, we have
	\begin{equation}\label{eq:boundsM}
	\begin{aligned}
		\sqrt{\mathcal{M}(t,k,\xi)}\Lambda_t^b(k,\xi)\leq C_{\mu}^{\f12}e^{1/2\delta_0\nu^{\f13}t}\Lambda_t^b(k,\xi)&\leq C_{\mu}^{\f12}e^{1/2\delta_0\nu^{\f13}t}\big(\Lambda_t^b(l,\eta)+\Lambda_t^b(k-l,\xi-\eta)\big),\\
		\sqrt{\mathcal{M}(t,k-l,\xi-\eta)}\Lambda_t^b(k-l,\xi-\eta)&\leq C_{\mu}^{\f12}e^{1/2\delta_0\nu^{\f13}t}\Lambda_t^b(k-l,\xi-l),\\
		\sqrt{\mathcal{M}(t,l,\eta)}\Lambda_t^b(l,\eta)&\leq C_{\mu}^{\f12}e^{1/2\delta_0\nu^{\f13}t}\Lambda_t^b(l,\eta).
	\end{aligned}
	\end{equation}
	Then we get
	\begin{align*}
		&|\mathcal{I}|\lesssim e^{\f12\delta_0\nu^{\f13}t}\sum_{k,l}\iint\Lambda_t^b(l,\eta)|\hat{v}^{2}(l,\eta)|
		 |\xi-\eta||\hat{f}(k-l,\xi-\eta)|\sqrt{\mathcal{M}(t,k,\xi)}\Lambda_t^b(k,\xi)|\hat{g}(k,\xi)|d\xi d\eta\\
		 &+e^{\f12\delta_0\nu^{\f13}t}\sum_{k,l}\iint\Lambda_t^b(k-l,\xi-\eta)|\hat{v}^{2}(l,\eta)|
		 |\xi-\eta||\hat{f}(k-l,\xi-\eta)|\sqrt{\mathcal{M}(t,k,\xi)}\Lambda_t^b(k,\xi)|\hat{g}(k,\xi)|d\xi d\eta\\
		 &=:e^{\f12\delta_0\nu^{\f13}t}(\mathcal{I}_{1}+\mathcal{I}_{2}).
	\end{align*}
	$\mathcal{I}_{1}$ corresponds to the interaction between the high mode of $v^{2}$ and the low mode of $f$. Moreover, $\mathcal{I}_{2}$ corresponds to the interaction between the low mode of $v^{2}$ and the high mode of $f$, which is weak for $t>T_0$ due to the inviscid damping effect.
	
	By Young's convolution inequality and taking $f=v^{2}$ in Lemma \ref{Lem:L1}, we have
	\begin{equation}
		\begin{aligned}
			|\mathcal{I}_{2}|\leq &C\|\hat{v}^{2}\|_{L^1}\|D_y\Lambda_t^bf\|_{L^2}\|\sqrt{\mathcal{M}_t}\Lambda_t^bg\|_{L^2}\\
			\leq& C(1+t)^{-1}\|\nabla\Lambda_t^bv^{2}\|_{L^2}\|D_y\Lambda_t^bf\|_{L^2}\|\sqrt{\mathcal{M}_t}\Lambda_t^bg\|_{L^2}.
		\end{aligned}
	\end{equation}
	For $\mathcal{I}_{1}$, by H\"older's inequality, we get $\mathcal{I}_{1}\leq (\mathcal{I}_{11})^{\f12}(\mathcal{I}_{12})^{\f12}$ with
	\begin{align*}
		\mathcal{I}_{11}:=&\sum_{k,l}\iint (l^2+\eta^2)|\Lambda_t^b(l,\eta)\hat{v}^{2}(l,\eta)|^2\Lambda_t^2(k-l,\xi-\eta)
		 |(\xi-\eta)\hat{f}(k-l,\xi-\eta)|^2d\xi d\eta,\\
		 \mathcal{I}_{12}:=&\sum_{k,l}\iint\frac{|\sqrt{\mathcal{M}(t,k,\xi)}\Lambda_t^b(k,\xi)\hat{g}(k,\xi)|^2}{(l^2+\eta^2)\Lambda_t^2(k-l,\xi-\eta)}d\xi d\eta.
	\end{align*}
	By Fubini's theorem, Plancherel's formula, and \eqref{eq:Mf}, we have
	\begin{align*}
		\mathcal{I}_{11}=\|\nabla \Lambda_t^bv^{2}\|_{L^2}^2\|\Lambda_tD_yf\|_{L^2}^2\leq C\nu^{-\f23}\|\nabla \Lambda_t^bv^{2}\|_{L^2}^2\|\sqrt{\mathcal{M}_t}\Lambda_t^bf\|_{L^2}^2.
	\end{align*}
	For $l\neq0$, by taking $a=|l|$, $s=1+|k-l|$, and $z=\xi+t(k-l)$ in \eqref{eq:as}, we get
	\begin{align*}
		\int \frac{d\eta}{(l^2+\eta^2)\Lambda_t^2(k-l,\xi-\eta)}=&\int \frac{d\eta}{(l^2+\eta^2)\big(1+(k-l)^2+(\xi-\eta+t(k-l))^2\big)}\\
		\leq&\frac{C}{|l|(1+|k-l|)}\f{1+|k-l|+|l|}{(1+|k-l|+|l|)^2+(\xi+t(k-l))^2}.
	\end{align*}
	Then according to the definition of $\Upsilon(t,k,\xi)$, we have
	\begin{equation}
		\begin{aligned}
			\mathcal{I}_{12}\leq &C\sum_{k,l\in\mathbb Z,l\neq0}\int\frac{1}{|l|(1+|k-l|)}\f{(1+|k-l|+|l|)|\sqrt{\mathcal{M}(t,k,\xi)}\Lambda_t^b(k,\xi)\hat{g}(k,\xi)|^2}{(1+|k-l|+|l|)^2+(\xi+t(k-l))^2}d\xi\\
			\leq& C\|\sqrt{\Upsilon_t\mathcal{M}_t}\Lambda_t^bg\|_{L^2}^2.
		\end{aligned}
	\end{equation}
	Thus we have
	\begin{align*}
		\mathcal{I}_{1}\leq (\mathcal{I}_{11})^{\f12}(\mathcal{I}_{12})^{\f12}\leq C\nu^{-\f13}\|\nabla \Lambda_t^bv^{2}\|_{L^2}\|\sqrt{\mathcal{M}_t}\Lambda_t^bf\|_{L^2}\|\sqrt{\Upsilon_t\mathcal{M}_t}\Lambda_t^bg\|_{L^2},
	\end{align*}
	and for $t\geq T_0=\nu^{-\f16}$ case, using the facts that $(1+t)^{-1}\leq T_0^{-1}=\nu^{\f16}$ and \eqref{eq:Mneq}, we have
		\begin{align*}
			|\mathcal{I}|\leq& e^{\f12\delta_0\nu^{\f13}t}(\mathcal{I}_{1}+\mathcal{I}_{2})
			\leq e^{\f12\delta_0\nu^{\f13}t}C\|\nabla \Lambda_t^bv^{2}\|_{L^2}\\
			&\times\big(\nu^{-\f13}\|\sqrt{\mathcal{M}_t}\Lambda_t^bf\|_{L^2}\|\sqrt{\Upsilon_t\mathcal{M}_t}\Lambda_t^bg\|_{L^2}+(1+t)^{-1}\|D_y\Lambda_t^bf\|_{L^2}\|\sqrt{\mathcal{M}_t}\Lambda_t^bg\|_{L^2}\big)\\
			\leq&C\|\nabla\sqrt{\mathcal{M}_t}\Lambda_t^bv^2\|_{L^2}\big(\nu^{-\f13}\|\sqrt{\mathcal{M}_t}\Lambda_t^bf\|_{L^2}\|\sqrt{\Upsilon_t\mathcal{M}_t}\Lambda_t^bg\|_{L^2}+\nu^{\f16}\|D_y\Lambda_t^bf\|_{L^2}\|\sqrt{\mathcal{M}_t}\Lambda_t^bg\|_{L^2}\big).
		\end{align*}
\end{proof}
The following lemma will be used to estimate the interaction in zero mode.
\begin{lemma}\label{Lem:u_0 comm-long}
    Let $v=(v^1,v^2)=\nabla^{\perp}(-\Delta)^{-1}w$, then for $b>2$, we have
    \begin{align*}
	&\left|\langle \sqrt{\mathcal{M}_t}\Lambda_t^b(v_0^1\pa_xf)-v_0^1\pa_x\big(\sqrt{\mathcal{M}_t}\Lambda_t^bf\big),\sqrt{\mathcal{M}_t}\Lambda_t^bg\rangle\right|\\
	\leq &C\|\Lambda_t^bw\|_{L^2}\||D_x|^{\f13}\sqrt{\mathcal{M}_t}\Lambda_t^bf\|_{L^2}\||D_x|^{\f13}\sqrt{\mathcal{M}_t}\Lambda_t^bg\|_{L^2}.
\end{align*}
\end{lemma}
\begin{proof}
	By Plancherel's formula and the fact $\hat{v}^1(0,\eta)=(i\eta)^{-1}\hat{w}(0,\eta)$, we have
	 \begin{align*}
		 &\left|\langle \sqrt{\mathcal{M}_t}\Lambda_t^b(v_0^1\pa_xf)-v_0^1\pa_x\big(\sqrt{\mathcal{M}_t}\Lambda_t^bf\big),\sqrt{\mathcal{M}_t}\Lambda_t^bg\rangle\right|\\
		 =&\sum_{k}\iint\big(\sqrt{\mathcal{M}(k,\xi)}\Lambda_t^b(k,\xi)-\sqrt{\mathcal{M}(k,\xi-\eta)}\Lambda_t^b(k,\xi-\eta)\big)ik\\
		 &\times  \hat{f}(k,\xi-\eta)\hat{v}^1(0,\eta)\sqrt{\mathcal{M}(k,\xi)}\Lambda_t^b(k,\xi)\overline{\hat{g}(k,\xi)}d\xi d\eta\\
		 =&\sum_{k}\iint\big(\sqrt{\mathcal{M}(k,\xi)}\Lambda_t^b(k,\xi)-\sqrt{\mathcal{M}(k,\xi-\eta)}\Lambda_t^b(k,\xi-\eta)\big)k\\
		 &\times  \hat{f}(k,\xi-\eta)\f{\hat{w}(0,\eta)}{\eta}\sqrt{\mathcal{M}(k,\xi)}\Lambda_t^b(k,\xi)\overline{\hat{g}(k,\xi)}d\xi d\eta.
	 \end{align*}
	 By Taylor's formula, we have
	 \begin{align*}
		 &\left|\sqrt{\mathcal{M}(k,\xi)}\Lambda_t^b(k,\xi)-\sqrt{\mathcal{M}(k,\xi-\eta)}\Lambda_t^b(k,\xi-\eta)\right|\\
		 \leq&\int_0^1\left|\pa_{\xi}\big(\sqrt{\mathcal{M}(k,\xi-s\eta)}\Lambda_t^b(k,\xi-s\eta)\big)\right||\eta|ds.
	 \end{align*}
	 According to Lemma \ref{de-M}, we know for $k\neq0$
	 \begin{equation}\label{Mm}
	 \begin{aligned}
		 \pa_{\xi}\big(\sqrt{\mathcal{M}(k,\xi)}\Lambda_t^b(k,\xi)\big)&\leq C(\nu^{\f13}|k|^{-\f13}+|k|^{-1})\sqrt{\mathcal{M}(k,\xi)}\Lambda_t^b(k,\xi)\\
		 &\leq C(\nu^{\f13}|k|^{-\f13}+|k|^{-1})e^{\f12\delta_0\nu^{\f13}t}\Lambda_t^b(k,\xi).
	 \end{aligned}
	 \end{equation}
	 Therefore, by Young's convolution inequality, \eqref{eq:boundsM}, and taking $f=|D_x|^{\f13}f$ in \eqref{eq:Mneq}, we have
		 \begin{align*}
			 &\left|\langle \sqrt{\mathcal{M}_t}\Lambda_t^b(v_0^1\pa_xf)-v_0^1\pa_x\big(\sqrt{\mathcal{M}_t}\Lambda_t^bf\big),\sqrt{\mathcal{M}_t}\Lambda_t^bg\rangle\right|\\
			 \leq &Ce^{\delta_0\nu^{\f13}t}\sum_{k\neq0}\iint(\nu^{\f13}|k|^{\f23}+1)\big(\Lambda_t^b(0,\eta)+\Lambda_t^b(k,\xi-\eta)\big)|\hat{f}(k,\xi-\eta)\hat{w}(0,\eta)|\\
		 &\times \Lambda_t^b(k,\xi)|\hat{g}(k,\xi)|d\xi d\eta\\
		 \leq&Ce^{\delta_0\nu^{\f13}t}\big(\nu^{\f13}\|\Lambda^b_tw_0\|_{L^2}\||D_x|^{\f13}\Lambda_t^bf\|_{L^2}\||D_x|^{\f13}\Lambda_t^bg\|_{L^2}+\|\Lambda^b_tw_0\|_{L^2}\|\Lambda_t^bf\|_{L^2}\|\Lambda_t^bg\|_{L^2}\big)\\
		 \leq&Ce^{\delta_0\nu^{\f13}t}\|\Lambda_t^bw_0\|_{L^2}\||D_x|^{\f13}\Lambda_t^bf\|_{L^2}\||D_x|^{\f13}\Lambda_t^bg\|_{L^2}\\
		 \leq& C\|\Lambda_t^bw\|_{L^2}\||D_x|^{\f13}\sqrt{\mathcal{M}_t}\Lambda_t^bf\|_{L^2}\||D_x|^{\f13}\sqrt{\mathcal{M}_t}\Lambda_t^bg\|_{L^2}.
		 \end{align*}
\end{proof}
\begin{lemma}\label{Lem:I31}
    Let $v=(v^1,v^2)=\nabla^{\perp}(-\Delta)^{-1}w$, then for $b>2$, we have for $t\geq \nu^{-\f16}$ that
    \begin{align*}
	&\left|\big\langle \sqrt{\mathcal{M}_t}\Lambda_t^b(v^1_{\neq}\pa_xf)-v^1_{\neq}\pa_x\big(\sqrt{\mathcal{M}_t}\Lambda_t^bf\big),\sqrt{\mathcal{M}_t}\Lambda_t^bf\big\rangle\right|\\
	\leq&C\|\sqrt{\mathcal{M}_t}\Lambda_t^bw\|_{L^2}\big(\||D_x|^{\f13}\sqrt{\mathcal{M}_t}\Lambda_t^bf\|_{L^2}^2+\nu^{\f23}\|D_y\Lambda_t^bf\|_{L^2}\||D_x|^{\f13}\sqrt{\mathcal{M}_t}\Lambda_t^bf\|_{L^2}\big) .
\end{align*}
\end{lemma}
\begin{proof}
By Plancherel's formula, we have
	 \begin{align*}
		 &\left|\big\langle \sqrt{\mathcal{M}_t}\Lambda_t^b(v^1_{\neq}\pa_xf)-v^1_{\neq}\pa_x\big(\sqrt{\mathcal{M}_t}\Lambda_t^bf\big),\sqrt{\mathcal{M}_t}\Lambda_t^bf\big\rangle\right|\\
		 =&\sum_{k,l}\iint\big(\sqrt{\mathcal{M}(t,k,\xi)}\Lambda_t^b(k,\xi)-\sqrt{\mathcal{M}(t,l,\eta)}\Lambda_t^b(l,\eta)\big)\hat{v}_{\neq}^1(k-l,\xi-\eta)\\
		 &\times il\hat{f}(l,\eta)\sqrt{\mathcal{M}(t,k,\xi)}\Lambda_t^b(k,\xi)\overline{\hat{f}(k,\xi)}d\xi d\eta=:\mathcal{J}.
	 \end{align*}
	 We split the range of $k$ and $l$ into two cases:
	 \begin{align*}
		 A_1=\Big\{(k,l)\in\mathbb Z^2: |k-l|\geq\f{|l|}{2}\ \mathrm{or}\ \min(|k|,|l|)\leq\nu^{-\f12} \Big\},\\
		 A_2=\Big\{(k,l)\in\mathbb Z^2: |k-l|<\f{|l|}{2}\ \mathrm{and}\ \min(|k|,|l|)>\nu^{-\f12} \Big\}.          
	 \end{align*}
	 Then we write $\mathcal{J}=\mathcal{J}_1+\mathcal{J}_2$, with
	 \begin{align*}
		 \mathcal{J}_j=&\sum_{(k,l)\in A_j}\iint\big(\sqrt{\mathcal{M}(t,k,\xi)}\Lambda_t^b(k,\xi)-\sqrt{\mathcal{M}(t,l,\eta)}\Lambda_t^b(l,\eta)\big)\hat{v}_{\neq}^1(k-l,\xi-\eta)\\
&\times il\hat{f}(l,\eta)\sqrt{\mathcal{M}(t,k,\xi)}\Lambda_t^b(k,\xi)\overline{\hat{f}(k,\xi)}d\xi d\eta.
	 \end{align*}
	 For $(k,l)\in A_2$, we apply Taylor's formula, and for $(k,l)\in A_1$, we use the bound $|\sqrt{\mathcal{M}(t,k,\xi)}\Lambda_t^b(k,\xi)-\sqrt{\mathcal{M}(t,l,\eta)}\Lambda_t^b(l,\eta)|\leq |\sqrt{\mathcal{M}(t,k,\xi)}\Lambda_t^b(k,\xi)|+|\sqrt{\mathcal{M}(t,l,\eta)}\Lambda_t^b(l,\eta)$. Notice that if $|k-l|\leq |l|/2$ and $\min(|k|,|l|)\leq \nu^{-\f12}$, then
	 \begin{align*}
		 |k|\geq\f{|l|}{2},\ \min(|k|,|l|)\geq\f{|l|}{2},\ |l|\leq 2\big(\min(|k|,|l|)|lk|\big)^{\f13}\leq 2\nu^{-\f16}|l|^{\f13}|k|^{\f13};
	 \end{align*}
	 if $|k-l|\geq|l|/2$, then $|l|\leq 2|k-l|$. Thus, if $(k,l)\in A_1$, then 
	 \begin{align*}
		 |l|\leq 2(|k-l|+\nu^{-\f16}|l|^{\f13}|k|^{\f13}).
	 \end{align*}
	 Then by \eqref{eq:boundsM} and Young's convolution inequality, for $t\geq T_0=\nu^{-\f16}$, we have
	 \begin{align*}
		 |\mathcal{J}_1|\leq &Ce^{\f12\delta_0\nu^{\f13}t}\sum_{(k,l)\in A_1}\iint\big(\Lambda_t^b(k-l,\xi-\eta)+\Lambda_t^b(l,\eta)\big)|\hat{v}_{\neq}^1(k-l,\xi-\eta)|\\
&\times |l||\hat{f}(l,\eta)|\sqrt{\mathcal{M}(t,k,\xi)}\Lambda_t^b(k,\xi)|\hat{f}(k,\xi)|d\xi d\eta\\
\leq& Ce^{\f12\delta_0\nu^{\f13}t}\sum_{(k,l)\in A_1}\iint\big(\Lambda_t^b(k-l,\xi-\eta)|l|+\Lambda_t^b(l,\eta)(|k-l|+\nu^{-\f16}|l|^{\f13}|k|^{\f13})\big)\\
&\times |\hat{v}_{\neq}^1(k-l,\xi-\eta)||\hat{f}(l,\eta)|\sqrt{\mathcal{M}(t,k,\xi)}\Lambda_t^b(k,\xi)|\hat{f}(k,\xi)|d\xi d\eta\\
\leq&Ce^{\f12\delta_0\nu^{\f13}t}\Big(\|\Lambda_t^bv^1_{\neq}\|_{L^2}\|\widehat{D_xf}\|_{L^1}\|\sqrt{\mathcal{M}_t}\Lambda_t^bf\|_{L^2}\\
&+\nu^{-\f16}\||D_x|^{\f13}\Lambda_t^bf\|_{L^2}\|\widehat{v^1_{\neq}}\|_{L^1}\||D_x|^{\f13}\sqrt{\mathcal{M}_t}\Lambda_t^bf\|_{L^2}\\
&+\|\widehat{D_xv_{\neq}^1}\|_{L^1}\|\Lambda_t^bf\|_{L^2}\|\sqrt{\mathcal{M}_t}\Lambda_t^bf\|_{L^2}\Big).
\end{align*}
Since $m>2$ , $m-\f23>1$, and $v^1_{\neq}=-\pa_y(-\Delta)^{-1}w_{\neq}$, we have
	 \begin{align*}
		 \|\widehat{D_xv^1_{\neq}}\|_{L^1}+\|\Lambda_t^bv^1_{\neq}\|_{L^2}\leq C\|D_x\Lambda_t^bv^1\|_{L^2}\leq C\|\Lambda_t^bw\|_{L^2}\leq C\|\sqrt{\mathcal{M}_t}\Lambda_t^bw\|_{L^2},\\
		 \|\widehat{D_xf}\|_{L^1}\leq C\|\Lambda_t^{m-\f23}D_xf\|_{L^2}\leq C\|\Lambda_t^{m}|D_x|^{\f13}f\|_{L^2}.
	 \end{align*}
Then, together with Lemma \ref{Lem:L1}, \eqref{eq:M}, \eqref{eq:Mneq}, and $t\geq T_0=\nu^{-\f16}$, we get
\begin{align*}
|\mathcal{J}_1|\leq& C\|e^{\f12\delta_0\nu^{\f13}t}\Lambda_t^{m}|D_x|^{\f13}f\|_{L^2}\|\Lambda_t^{m}w\|_{L^2}\|\sqrt{\mathcal{M}_t}\Lambda_t^bf\|_{L^2}\\
&+C\|e^{\f12\delta_0\nu^{\f13}t}\Lambda_t^{m}|D_x|^{\f13}f\|_{L^2}\|\Lambda_t^bw\|_{L^2}\||D_x|^{\f13}\sqrt{\mathcal{M}_t}\Lambda_t^bf\|_{L^2}\\
\leq&C\|\sqrt{\mathcal{M}_t}\Lambda_t^bw\|_{L^2}\||D_x|^{\f13}\sqrt{\mathcal{M}_t}\Lambda_t^bf\|_{L^2}^2.
	 \end{align*}

	 Next, we begin to estimate $\mathcal{J}_2$. For $(k,l)\in A_2$, we have $|k-l|\leq |l|/2<|l|$, $kl>0$, $|k|>\nu^{-\f12}$, and $|l|>\nu^{-\f12}$. By Taylor's formula, we get
	 \begin{align*}
		 &\left|\sqrt{\mathcal{M}(t,k,\xi)}\Lambda_t^b(k,\xi)-\sqrt{\mathcal{M}(t,l,\eta)}\Lambda_t^b(l,\eta)\right|\\
		 \leq& \int_0^1\left|\pa_{\xi}\big(\sqrt{\mathcal{M}(t,k,\eta+s(\xi-\eta))}\Lambda_t^b(k,\eta+s(\xi-\eta))\big)\right||\xi-\eta|ds\\
		 &+\int_0^1\left|\pa_{k}\big(\sqrt{\mathcal{M}(t,l+s(k-l),\eta)}\Lambda_t^b(l+s(k-l),\eta)\big)\right||k-l|ds.
	 \end{align*}
	 A direct calculation gives
	 \begin{align*}
		 |\pa_k\Lambda_t^b(k,\xi)|\leq C\Lambda_t^{m-2}(k,\xi)\big(|k|+|\xi+kt||t|\big),
		 \end{align*}
		since $|t|\leq |k|^{-1}(|\xi|+\Lambda_t(k,\xi))|$, we know
		\begin{align*}
		 |\pa_k\Lambda_t^b(k,\xi)|\leq C\Lambda_t^{m}(k,\xi)\big(|k|^{-1}+|k|^{-2}|\xi|\big).
	 \end{align*}
	 Then by Lemma \ref{de-M} and \eqref{eq:boundsM}, for $|k|>\nu^{-\f12}$, we have
	 \begin{align*}
		 \left|\pa_k\big(\sqrt{\mathcal{M}(t,k,\xi)}\Lambda_t^b(k,\xi)\big)\right|\leq &C\big(|k|^{-1}+|k|^{-2}|\xi|+\nu^{\f12}|k|^{-1}|\xi|\big)\sqrt{\mathcal{M}(t,k,\xi)}\Lambda_t^b(k,\xi)\\
		 \leq&C\big(|k|^{-1}+\nu^{\f12}|k|^{-1}|\xi|\big)e^{\f12\delta_0\nu^{\f13}t}\Lambda_t^b(k,\xi).
	 \end{align*}
	 For $|k|>\nu^{-\f12}$, we get by \eqref{Mm} that
	 \begin{align*}
		 \pa_{\xi}\big(\sqrt{\mathcal{M}(k,\xi)}\Lambda_t^b(k,\xi)\big)\leq C\nu^{\f13}|k|^{-\f13}e^{\f12\delta_0\nu^{\f13}t}\Lambda_t^b(k,\xi).
	 \end{align*}
	 Thus, for $(k,l)\in A_2$ and $s\in[0,1]$, by \eqref{eq:boundsM} and \eqref{eq:Mneq}, we have $kl>0$, $|l+s(k-l)|\geq\min(|k|,|l|)>\nu^{-\f12}$, $|l+s(k-l)|\geq|l|-|k-l|\geq|l|/2$, and
	 \begin{align*}
		 &\left|\sqrt{\mathcal{M}(t,k,\xi)}\Lambda_t^b(k,\xi)-\sqrt{\mathcal{M}(t,l,\eta)}\Lambda_t^b(l,\eta)\right|\\
		 \leq& Ce^{\f12\delta_0\nu^{\f13}t}\int_0^1\nu^{\f13}|k|^{-\f13}\Lambda_t^b(k,\eta+s(\xi-\eta))|\xi-\eta|ds\\
		 &+Ce^{\f12\delta_0\nu^{\f13}t}\int_0^1\f{1+\nu^{\f12}|\eta|}{|l+s(k-l)|}\Lambda_t^b(l+s(k-l),\eta)|k-l|ds\\
		 \leq&Ce^{\f12\delta_0\nu^{\f13}t}\big(\nu^{\f13}|k|^{-\f13}|\xi-\eta|+|l|^{-1}(1+\nu^{\f12}|\eta|)|k-l|\big)\big(\Lambda_t^b(l,\eta)+\Lambda_t^b(k-l,\xi-\eta)\big).
	 \end{align*}
	 Then we have
	 \begin{align*}
		 |\mathcal{J}_2|\leq&Ce^{\f12\delta_0\nu^{\f13}t}\sum_{(k,l)\in A_2}\iint\big(\nu^{\f13}|k|^{-\f13}|\xi-\eta|+|l|^{-1}(1+\nu^{\f12}|\eta|)|k-l|\big)\left|\hat{v}^1_{\neq}(k-l,\xi-\eta)\right|\\
		 &\times\big(\Lambda_t^b(l,\eta)+\Lambda_t^b(k-l,\xi-\eta)\big) \left|l\hat{f}(l,\eta)\right|\sqrt{\mathcal{M}(t,k,\xi)}\Lambda_t^b(k,\xi)\left|\hat{f}(k,\xi)\right|d\xi d\eta.
		 \end{align*}
	Since $v^1=-\pa_y(-\Delta)^{-1}w$, we have
	 \begin{align*}
		 \big(|k-l|+|\xi-\eta|\big)\left|\hat{v}^1(k-l,\xi-\eta)\right|\leq 2\left|\hat{w}(k-l,\xi-\eta)\right|.
	 \end{align*}
	 And for $(k,l)\in A_2$, we have that $|k-l|\leq|l|/2\leq|k|\leq3|l|/2$, and $\nu^{\f13}|k|^{-\f13}|l|\leq C\nu^{\f13}|k|^{\f13}|l|^{\f13}$, moreover, $\min(|k|,|l|)>\nu^{-\f12}$, $1<\nu^{\f13}|k|^{\f13}|l|^{\f13}$, and $\nu^{\f12}|\eta|\leq\nu^{\f23}|k|^{\f13}|\eta|$. Thus we get
		 \begin{align*}
		 |\mathcal{J}_2|\leq& Ce^{\f12\delta_0\nu^{\f13}t}\sum_{(k,l)\in A_2}\iint\big(\nu^{\f13}|k|^{\f13}|l|^{\f13}+\nu^{\f23}|k|^{\f13}|\eta|\big)\left|\hat{w}_{\neq}(k-l,\xi-\eta)\right|\\
		 &\times\big(\Lambda_t^b(l,\eta)+\Lambda_t^b(k-l,\xi-\eta)\big) \left|\hat{f}(l,\eta)\right|\sqrt{\mathcal{M}(t,k,\xi)}\Lambda_t^b(k,\xi)\left|\hat{f}(k,\xi)\right|d\xi d\eta\\
		 \leq& Ce^{\delta_0\nu^{\f13}t}\|\hat{w}_{\neq}\|_{L^1}\big(\nu^{\f13}\||D_x|^{\f13}\Lambda_t^bf\|_{L^2}^2+\nu^{\f23}\|D_y\Lambda_t^bf\|_{L^2}\||D_x|^{\f13}\Lambda_t^bf\|_{L^2}\big)\\
		 &+Ce^{\delta_0\nu^{\f13}t}\|\Lambda_t^bw_{\neq}\|_{L^2}\big(\nu^{\f13}\|\widehat{|D_x|^{\f13}f}\|_{L^1}\||D_x|^{\f13}\Lambda_t^bf\|_{L^2}+\nu^{\f23}\|\widehat{D_yf}\|_{L^1}\||D_x|^{\f13}\Lambda_t^bf\|_{L^2}\big).
		 \end{align*}
		 Then by the fact that for $g=w_{\neq}$ or $g=|D_x|^{\f13}f$ or $g=D_yf$, $\|\hat{g}\|_{L^1}\leq C\|\Lambda_t^bg\|_{L^2}$, we have
		 \begin{align*}
		 |\mathcal{J}_2|\leq&C\|\sqrt{\mathcal{M}_t}\Lambda_t^bw\|_{L^2}\big(\nu^{\f13}\||D_x|^{\f13}\sqrt{\mathcal{M}_t}\Lambda_t^bf\|_{L^2}^2+\nu^{\f23}\|D_y\Lambda_t^bf\|_{L^2}\||D_x|^{\f13}\sqrt{\mathcal{M}_t}\Lambda_t^bf\|_{L^2}\big).
	 \end{align*}

	 Thus, we know
	 \begin{align*}
		 |\mathcal{J}|\leq C\|\sqrt{\mathcal{M}_t}\Lambda_t^bw\|_{L^2}\big(\||D_x|^{\f13}\sqrt{\mathcal{M}_t}\Lambda_t^bf\|_{L^2}^2+\nu^{\f23}\|D_y\Lambda_t^bf\|_{L^2}\||D_x|^{\f13}\sqrt{\mathcal{M}_t}\Lambda_t^bf\|_{L^2}\big).
		 \end{align*}
\end{proof}
\section{Estimates of $\omega^{\mathrm{i}}$ and $\th^{\mathrm{i}}$}\label{sec-i}
In this section, we study $(\omega^{\mathrm{i}}, \th^{\mathrm{i}})$. The main idea is to show the equation of $\theta^{\mathrm{i}}$ allows us to propagate more derivatives in the $x$ direction. 
\begin{proposition}\label{prop: short-time-i}
    For $m>2$ and $m_1\geq1$, there exists $\ep_0>0$ small and independent of $\nu$ such that if $0<\nu\leq1$, $\|\omega_{\mathrm{in}}\|_{H^m}\leq \ep_0\nu^{\f13}$, and $\|\langle\pa_x\rangle^{m_1}\th_{\mathrm{in}}\|_{H^m}\leq\ep_0\nu^{\frac{2}{3}}$, then it holds for $T_0=\nu^{-\f16}$ and $t\in[0,T_0]$ that
    \begin{align*}
        \|\Lambda_t^m\omega^{\mathrm{i}}(t)\|_{L^2}+\nu^{-\f13}\|\langle \partial_x\rangle^{m_1}\Lambda_t^{m}\theta^{\mathrm{i}}\|_{L^2}\leq C\ep_0 \nu^{\f13}.
    \end{align*}
\end{proposition}
\begin{proof}
Applying the operator $\Lambda_t^m$ to $\eqref{eq:i}_1$ and using the fact that $\Lambda_t^m$ commutes with $\pa_t+y\pa_x$, we have
\begin{equation}\label{idw^i}
    \pa_t\Lambda_t^m\omega^{\mathrm{i}}+y\pa_x\Lambda_t^m\omega^{\mathrm{i}}-\nu\Delta\Lambda_t^m\omega^{\mathrm{i}}+\Lambda_t^m\big(u^{\mathrm{i}}\cdot\nabla\omega^{\mathrm{i}}\big)=\pa_x\Lambda_t^m\theta^{\mathrm{i}}.
\end{equation}
Taking inner product of above equation with $\Lambda_t^m\omega^{\mathrm{i}}$, we have
\begin{align*}
    \frac{d}{dt}\|\Lambda_t^m\omega^{\mathrm{i}}\|_{L^2}^2+2\nu\|\nabla\Lambda_t^m\omega^{\mathrm{i}}\|_{L^2}^2+2\langle \Lambda_t^m\big(u^{\mathrm{i}}\cdot\nabla\omega^{\mathrm{i}}\big),\Lambda_t^m\omega^{\mathrm{i}}\rangle=2\langle \pa_x\Lambda_t^m\theta^{\mathrm{i}},\Lambda_t^m\omega^{\mathrm{i}}\rangle.
\end{align*}
By taking $b=d=m$, $v=u^{\mathrm{i}}$ and $w=f=\omega^{\mathrm{i}}$ in Lemma \ref{Lem: transport-shorttime}, we have 
\begin{align*}
    |\langle \Lambda_t^m\big(u^{\mathrm{i}}\cdot\nabla\omega^{\mathrm{i}}\big),\Lambda_t^m\omega^{\mathrm{i}}\rangle|\lesssim (1+t)\|\Lambda_t^m\omega^{\mathrm{i}}\|_{L^2}^3,
\end{align*}
and we also know
\begin{align*}
    |\langle \pa_x\Lambda_t^m\theta^{\mathrm{i}},\Lambda_t^m\omega^{\mathrm{i}}\rangle|\lesssim \|\pa_x\Lambda_t^m\theta^{\mathrm{i}}\|_{L^2}\|\Lambda_t^m\omega^{\mathrm{i}}\|_{L^2},
\end{align*}
then we have
\begin{equation}\label{eq:Lambda-omega}
\begin{aligned}
    \frac{d}{dt}\|\Lambda_t^m\omega^{\mathrm{i}}\|_{L^2}^2+2\nu\|\nabla\Lambda_t^m\omega^{\mathrm{i}}\|_{L^2}^2\leq C(1+t)\|\Lambda_t^m\omega^{\mathrm{i}}\|_{L^2}^3+K\nu^{-\f12}\|\pa_x\Lambda_t^m\theta^{\mathrm{i}}\|_{L^2}^2+\nu^{\f12}\|\Lambda_t^m\omega^{\mathrm{i}}\|_{L^2}^2.
\end{aligned}
\end{equation}
And applying the operator $\langle\pa_x\rangle^{m_1}\Lambda_t^m$ to $\eqref{eq:i}_2$ and using the fact that $\Lambda_t^m$ commutes with $\pa_t+y\pa_x$, we have
\begin{align*}
    \pa_t\langle\pa_x\rangle^{m_1}\Lambda_t^m\th^{\mathrm{i}}+y\pa_x\langle\pa_x\rangle^{m_1}\Lambda_t^m\th^{\mathrm{i}}-\nu\Delta\langle\pa_x\rangle^{m_1}\Lambda_t^m\th^{\mathrm{i}}+\Lambda_t^m\big(u^{\mathrm{i},1}_0\langle\pa_x\rangle^{m_1}\pa_x\th^{\mathrm{i}}\big)=0,
\end{align*}
taking inner product of above equation with $\langle\pa_x\rangle^{m_1}\Lambda_t^m\th^{\mathrm{i}}$, we have
\begin{align*}
    \frac{d}{dt}\|\langle\pa_x\rangle^{m_1}\Lambda_t^m\th^{\mathrm{i}}\|_{L^2}^2+2\nu\|\nabla \langle\pa_x\rangle^{m_1}\Lambda_t^m\th^{\mathrm{i}}\|_{L^2}^2=-\langle \Lambda_t^m\big(u^{\mathrm{i},1}_0\pa_x\langle\pa_x\rangle^{m_1}\th^{\mathrm{i}}\big), \Lambda_t^m\langle\pa_x\rangle^{m_1}\th^{\mathrm{i}}\rangle.
\end{align*}
Since $\th^{\mathrm{i}}=\th^{\mathrm{i}}_0+\th^{\mathrm{i}}_{\neq}$, for zero mode part, we directly have
\begin{align*}
    \|\th^{\mathrm{i}}_0\|_{H^m}\leq \|\th_{\mathrm{in}}\|_{H^m},
\end{align*}
by taking $b=d=m$, $w=\omega_0^{\mathrm{i}}$ and $f=\langle\pa_x\rangle^{m_1}\th^{\mathrm{i}}$ in Lemma \ref{Lem: transport-shorttime}, for the nonzero parts we have
\begin{align*}
    \left|\langle \Lambda_t^m\big(u^{\mathrm{i},1}_0\pa_x\langle\pa_x\rangle^{m_1}\th^{\mathrm{i}}_{\neq}\big), \Lambda_t^m\langle\pa_x\rangle^{m_1}\th^{\mathrm{i}}_{\neq}\rangle\right|\leq C\|\Lambda_t^m\omega_0^{\mathrm{i}}\|_{L^2}\|\Lambda_t^m\langle\pa_x\rangle^{m_1}\th^{\mathrm{i}}_{\neq}\|_{L^2}^2,
\end{align*}
combined with \eqref{eq:Lambda-omega}, and then the result follows by Gr\"onwall's inequality and the fact that $t\leq \nu^{-\f16}$.
\end{proof}
\begin{proposition}\label{prop: long-time i}
    For $m>2$, $m_1\geq 1$, there exists $\delta>0$ independent of $\nu$ such that if $0<\nu\leq 1$, $t\geq T_0=\nu^{-\f16}$, $\|\sqrt{\mathcal{M}_t}\Lambda_t^m\omega^{\mathrm{i}}(t)\|_{L^2}\leq\delta\nu^{\f13}$, and $\|\langle\pa_x\rangle^{m_1}\sqrt{\mathcal{M}_t}\Lambda_t^m\th^{\mathrm{i}}(t)\|_{L^2}\leq\delta\nu^{\f23}$, then it holds for a large constant $K$ independent of $\nu$ that
    \begin{align*}
        \frac{d}{dt}\left(\|\sqrt{\mathcal{M}_t}\Lambda_t^m\omega^{\mathrm{i}}(t)\|_{L^2}^2
        +K\nu^{-\f23}\|\langle\pa_x\rangle^{m_1}\sqrt{\mathcal{M}_t}\Lambda_t^m\th_{\neq}^{\mathrm{i}}(t)\|_{L^2}^2\right)+\mathcal{CK}(\omega^{\mathrm{i}})+K\nu^{-\f23}\mathcal{CK}(\theta^{\mathrm{i}})\leq 0,
    \end{align*}
    where
    \begin{align*}
        \mathcal{CK}(\omega^{\mathrm{i}})=&\frac{\delta_0}{2}\Big(\nu\left\|(D_x,D_y)\sqrt{\mathcal{M}_t}\Lambda_t^m\omega^{\mathrm{i}}\right\|_{L^2}^2+\nu^{\f13}\left\||D_x|^{\frac{1}{3}}\sqrt{\mathcal{M}_t}\Lambda_t^m\omega^{\mathrm{i}}\right\|_{L^2}^2\\
         &+\left\|\nabla\sqrt{\mathcal{M}_t}\Lambda_t^mu_{\neq}^{\mathrm{i}}\right\|_{L^2}^2+\left\|\sqrt{\mathcal{M}_t\Upsilon_t}\Lambda_t^m\omega^{\mathrm{i}}\right\|_{L^2}^2\Big),
    \end{align*}
    and
    \begin{align*}
        \mathcal{CK}(\th^{\mathrm{i}})=\frac{\delta_0}{2}\Big(\nu\left\|(D_x,D_y)\sqrt{\mathcal{M}_t}\Lambda_t^m\langle\partial_x\rangle^{m_1}\th^{\mathrm{i}}\right\|_{L^2}^2+\nu^{\f13}\left\||D_x|^{\frac{1}{3}}\sqrt{\mathcal{M}_t}\Lambda_t^m\langle\partial_x\rangle^{m_1}\th^{\mathrm{i}}\right\|_{L^2}^2\Big).
    \end{align*}
\end{proposition}
 \begin{proof}
     Taking inner product of $\eqref{idw^i}$ with $\mathcal{M}_t\Lambda_t^m\omega^{\mathrm{i}}$, due to the fact that $\mathcal{M}_t$ is self-adjoint and $y\pa_x$ is skew-adjoint, we get
     \begin{align*}
         2\mathrm{Re}\langle y\pa_x\omega^{\mathrm{i}},\mathcal{M}_t\omega^{\mathrm{i}}\rangle=\langle [\mathcal{M}_t,y\pa_x]\omega^{\mathrm{i}},\omega^{\mathrm{i}}\rangle=\langle k\pa_{\xi}\mathcal{M}(t,D)\omega^{\mathrm{i}},\omega^{\mathrm{i}}\rangle,\\
         2\mathrm{Re}\langle \pa_t\omega^{\mathrm{i}},\mathcal{M}_t\omega^{\mathrm{i}}\rangle=\pa_t\langle\omega^{\mathrm{i}},\mathcal{M}_t\omega^{\mathrm{i}}\rangle-\langle \omega^{\mathrm{i}},\pa_t\mathcal{M}(t,D)\omega^{\mathrm{i}}\rangle.
     \end{align*}
     Thus we infer that
     \begin{align*}
         &\f{d}{dt}\|\sqrt{\mathcal{M}_t}\Lambda_t^m\omega^{\mathrm{i}}\|_{L^2}^2+\langle (k\pa_{\xi}-\pa_t)\mathcal{M}(t,D)\Lambda_t^m\omega^{\mathrm{i}},\Lambda_t^m\omega^{\mathrm{i}}\rangle+2\nu\|(D_x,D_y)\sqrt{\mathcal{M}_t}\Lambda_t^m\omega^{\mathrm{i}}\|_{L^2}^2\\
         &+2\mathrm{Re}\langle \Lambda_t^m(u^{\mathrm{i}}\cdot\nabla\omega^{\mathrm{i}}),\mathcal{M}_t\Lambda_t^m\omega^{\mathrm{i}}\rangle=2\mathrm{Re}\langle \pa_x\Lambda_t^m\th^{\mathrm{i}},\mathcal{M}_t\Lambda_t^m\omega^{\mathrm{i}}\rangle.
     \end{align*}
     Then by taking $f=\Lambda_t^m\omega^{\mathrm{i}}$ in \eqref{energy-dissipation} and $u_{\neq}^{\mathrm{i},2}=\pa_x(-\Delta)^{-1}\omega_{\neq}^{\mathrm{i}}$, we have
     \begin{equation}\label{eq:equation}
     \begin{aligned}
         &\f{d}{dt}\|\sqrt{\mathcal{M}_t}\Lambda_t^m\omega^{\mathrm{i}}\|_{L^2}^2+\frac{\delta_0}{2}\Big(\nu\left\|(D_x,D_y)\sqrt{\mathcal{M}_t}\Lambda_t^m\omega^{\mathrm{i}}\right\|_{L^2}^2+\nu^{\f13}\left\||D_x|^{\frac{1}{3}}\sqrt{\mathcal{M}_t}\Lambda_t^m\omega^{\mathrm{i}}\right\|_{L^2}^2\\
         &+\left\|\nabla\sqrt{\mathcal{M}_t}\Lambda_t^mu_{\neq}^{\mathrm{i}}\right\|_{L^2}^2+\left\|\sqrt{\mathcal{M}_t\Upsilon_t}\Lambda_t^m\omega^{\mathrm{i}}\right\|_{L^2}^2\Big)\\
         \leq&-2\mathrm{Re}\langle \Lambda_t^m(u^{\mathrm{i}}\cdot\nabla\omega^{\mathrm{i}}),\mathcal{M}_t\Lambda_t^m\omega^{\mathrm{i}}\rangle+2\mathrm{Re}\langle \pa_x\Lambda_t^m\th^{\mathrm{i}},\mathcal{M}_t\Lambda_t^m\omega^{\mathrm{i}}\rangle=:-I^1+I^2.
     \end{aligned}
     \end{equation}
     Since $\sqrt{\mathcal{M}_t}$ is also self-adjoint and $\mathrm{div}u^{\mathrm{i}}=0$, we know
     \begin{align*}
         \langle \Lambda_t^m(u^{\mathrm{i}}\cdot\nabla\omega^{\mathrm{i}}),\mathcal{M}_t\Lambda_t^m\omega^{\mathrm{i}}\rangle=\big\langle \sqrt{\mathcal{M}_t}\Lambda_t^m(u^{\mathrm{i}}\cdot\nabla\omega^{\mathrm{i}})-u^{\mathrm{i}}\cdot\nabla\big(\sqrt{\mathcal{M}_t}\Lambda_t^m\omega^{\mathrm{i}}\big),\sqrt{\mathcal{M}_t}\Lambda_t^m\omega^{\mathrm{i}}\big\rangle.
     \end{align*}
     Since $u^{\mathrm{i}}=P_0u^{\mathrm{i}}+P_{\neq}u^{\mathrm{i}}=(u_0^{\mathrm{i},1},0)+(u_{\neq}^{\mathrm{i},1},u_{\neq}^{\mathrm{i},2})$, we can decompose $I^1=I_1^1+I_2^1+I_3^1$ with
     \begin{equation}
     \begin{aligned}
         I_1^1:=\big\langle \sqrt{\mathcal{M}_t}\Lambda_t^m(u_{\neq}^{\mathrm{i},2}\pa_y\omega^{\mathrm{i}})-u_{\neq}^{\mathrm{i},2}\pa_y\big(\sqrt{\mathcal{M}_t}\Lambda_t^m\omega^{\mathrm{i}}\big),\sqrt{\mathcal{M}_t}\Lambda_t^m\omega^{\mathrm{i}}\big\rangle,\\
         I_2^1:=\big\langle \sqrt{\mathcal{M}_t}\Lambda_t^m(u_{0}^{\mathrm{i},1}\pa_x\omega^{\mathrm{i}})-u_{0}^{\mathrm{i},1}\pa_x\big(\sqrt{\mathcal{M}_t}\Lambda_t^m\omega^{\mathrm{i}}\big),\sqrt{\mathcal{M}_t}\Lambda_t^m\omega^{\mathrm{i}}\big\rangle,\\
         I_3^1:=\big\langle \sqrt{\mathcal{M}_t}\Lambda_t^m(u_{\neq}^{\mathrm{i},1}\pa_x\omega^{\mathrm{i}})-u_{\neq}^{\mathrm{i},1}\pa_x\big(\sqrt{\mathcal{M}_t}\Lambda_t^m\omega^{\mathrm{i}}\big),\sqrt{\mathcal{M}_t}\Lambda_t^m\omega^{\mathrm{i}}\big\rangle.
         \end{aligned}
     \end{equation}
     Next, we begin to estimate $I^1$ term by term.
    By taking $w=f=g=\omega^{\mathrm{i}}$ in Lemma \ref{Lem:I11}, we can directly obtain 
         \begin{align*}
             |I_1^1|\leq C\|\nabla\sqrt{\mathcal{M}_t}\Lambda_t^mu_{\neq}^{\mathrm{i},2}\|_{L^2}\|\sqrt{\mathcal{M}_t}\Lambda_t^m\omega^{\mathrm{i}}\|_{L^2}\big(\nu^{-\f13}\|\sqrt{\Upsilon_t\mathcal{M}_t}\Lambda_t^m\omega^{\mathrm{i}}\|_{L^2}+\nu^{\f16}\|D_y\Lambda_t^m\omega^{\mathrm{i}}\|_{L^2}\big).
         \end{align*}
        By taking $w=f=g=\omega^{\mathrm{i}}$ in Lemma \ref{Lem:u_0 comm-long}, we get
             \begin{align*}
                 |I_2^1|\leq C\|\Lambda^m_t\omega^{\mathrm{i}}\|_{L^2}\||D_x|^{\f13}\sqrt{\mathcal{M}_t}\Lambda_t^m\omega^{\mathrm{i}}\|_{L^2}^2.
             \end{align*}
         By taking $w=\omega_{\neq}^{\mathrm{i}}$ and $f=\omega^{\mathrm{i}}$ in Lemma \ref{Lem:I31}, we know
         \begin{align*}
             |I_3^1|\leq C\|\sqrt{\mathcal{M}_t}\Lambda_t^m\omega_{\neq}^{\mathrm{i}}\|_{L^2}\big(\||D_x|^{\f13}\sqrt{\mathcal{M}_t}\Lambda_t^m\omega^{\mathrm{i}}\|_{L^2}^2+\nu^{\f23}\|D_y\Lambda_t^m\omega^{\mathrm{i}}\|_{L^2}\||D_x|^{\f13}\sqrt{\mathcal{M}_t}\Lambda_t^m\omega^{\mathrm{i}}\|_{L^2}\big).
             \end{align*}
     Next for $I^2$, by a direct calculation, we have
     \begin{align*}
         |I^2|\leq C\|\pa_x\sqrt{\mathcal{M}_t}\Lambda_t^m\th^{\mathrm{i}}\|_{L^2}\|\sqrt{\mathcal{M}_t}\Lambda_t^m\omega_{\neq}^{\mathrm{i}}\|_{L^2}.
     \end{align*}
     Summing up, we have
     \begin{align*}
         |I^1|+|I^2|\leq&C\|\sqrt{\mathcal{M}_t}\Lambda_t^m\omega^{\mathrm{i}}\|_{L^2}\Big(\nu^{-\f13}\|\nabla\sqrt{\mathcal{M}_t}\Lambda_t^mu_{\neq}^{\mathrm{i},2}\|_{L^2}^2+\nu^{-\f13}\|\sqrt{\Upsilon_t\mathcal{M}_t}\Lambda_t^m\omega^{\mathrm{i}}\|_{L^2}^2\\
         &+\nu^{\f23}\|D_y\Lambda_t^m\omega^{\mathrm{i}}\|_{L^2}^2+\||D_x|^{\f13}\sqrt{\mathcal{M}_t}\Lambda_t^m\omega^{\mathrm{i}}\|_{L^2}^2\Big)+C\|\pa_x\sqrt{\mathcal{M}_t}\Lambda_t^m\th^{\mathrm{i}}\|_{L^2}\|\sqrt{\mathcal{M}_t}\Lambda_t^m\omega_{\neq}^{\mathrm{i}}\|_{L^2}.
     \end{align*}
     Then, we infer from \eqref{eq:equation} that
     \begin{equation}\label{eq:equation-w^i}
     \begin{aligned}
         &\f{d}{dt}\|\sqrt{\mathcal{M}_t}\Lambda_t^m\omega^{\mathrm{i}}\|_{L^2}^2+\frac{\delta_0}{2}\Big(\nu\left\|(D_x,D_y)\sqrt{\mathcal{M}_t}\Lambda_t^m\omega^{\mathrm{i}}\right\|_{L^2}^2+\nu^{\f13}\left\||D_x|^{\frac{1}{3}}\sqrt{\mathcal{M}_t}\Lambda_t^m\omega^{\mathrm{i}}\right\|_{L^2}^2\\
         &+\left\|\nabla\sqrt{\mathcal{M}_t}\Lambda_t^mu_{\neq}^{\mathrm{i}}\right\|_{L^2}^2+\left\|\sqrt{\mathcal{M}_t\Upsilon_t}\Lambda_t^m\omega^{\mathrm{i}}\right\|_{L^2}^2\Big)\\
         \leq&C_{\omega}\|\sqrt{\mathcal{M}_t}\Lambda_t^m\omega^{\mathrm{i}}\|_{L^2}\Big(\nu^{-\f13}\|\nabla\sqrt{\mathcal{M}_t}\Lambda_t^mu_{\neq}^{\mathrm{i},2}\|_{L^2}^2+\nu^{-\f13}\|\sqrt{\Upsilon_t\mathcal{M}_t}\Lambda_t^m\omega^{\mathrm{i}}\|_{L^2}^2+\nu^{\f23}\|D_y\Lambda_t^m\omega^{\mathrm{i}}\|_{L^2}^2\\
         &+\||D_x|^{\f13}\sqrt{\mathcal{M}_t}\Lambda_t^m\omega^{\mathrm{i}}\|_{L^2}^2\Big)+C_{\omega}\|\sqrt{\mathcal{M}_t}\Lambda_t^m\omega_{\neq}^{\mathrm{i}}\|_{L^2}\|\pa_x\sqrt{\mathcal{M}_t}\Lambda_t^m\th^{\mathrm{i}}\|_{L^2}.
     \end{aligned}
     \end{equation}
     A direct calculation gives that 
     \begin{align*}
         &\f{d}{dt}\|\sqrt{\mathcal{M}_t}\Lambda_t^m\langle\partial_x\rangle^{m_1}\th^{\mathrm{i}}\|_{L^2}^2+\frac{\delta_0}{2}\Big(\nu\left\|(D_x,D_y)\sqrt{\mathcal{M}_t}\Lambda_t^m\langle\partial_x\rangle^{m_1}\th^{\mathrm{i}}\right\|_{L^2}^2+\nu^{\f13}\left\||D_x|^{\frac{1}{3}}\sqrt{\mathcal{M}_t}\Lambda_t^m\langle\partial_x\rangle^{m_1}\th^{\mathrm{i}}\right\|_{L^2}^2\Big)\\
         \leq&-2\mathrm{Re}\langle \Lambda_t^m(u^{\mathrm{i},1}_0\partial_x\langle\partial_x\rangle^{m_1}\th^{\mathrm{i}})-u^{\mathrm{i},1}_0\partial_x\Lambda_t^m(\langle\partial_x\rangle^{m_1}\th^{\mathrm{i}}),\mathcal{M}_t\Lambda_t^m\langle\partial_x\rangle^{m_1}\th^{\mathrm{i}}\rangle=:I_{\th}.
     \end{align*}
     By taking $w=\omega^{\mathrm{i}}$, $f=g=\langle\partial_x\rangle^{m_1}\th^{\mathrm{i}}$ in Lemma \ref{Lem:u_0 comm-long}, we directly have
     \begin{align*}
         |I_{\th}|\leq C_{\th}\|\Lambda_t^m\omega^{\mathrm{i}}\|_{L^2}\||D_x|^{\f13}\sqrt{\mathcal{M}_t}\Lambda_t^m\langle\partial_x\rangle^{m_1}\th^{\mathrm{i}}\|_{L^2}^2.
     \end{align*}
     Then, combining with \eqref{eq:equation-w^i}, we obtain the results by taking $C_{\th}\delta=C_{\omega}\delta=\f{\delta_0}{4}$ and $K$ sufficiently large.
 \end{proof}

We conclude this section by summarizing the estimates of $\omega^{\mathrm{i}},\theta^{\mathrm{i}}$. By the standard bootstrap argument, we have for $t\leq \nu^{-\f16}$
\begin{align*}
    \|\Lambda_t^m\omega^{\mathrm{i}}\|_{L^2}\leq C\epsilon_0\nu^{\f13},\quad \|\langle\partial_x\rangle\Lambda_t^m\theta^{\mathrm{i}}\|_{L^2}\leq C\epsilon_0\nu^{\f23},
\end{align*}
and for $t\geq \nu^{-\f16}$
\begin{align*}
    \|\sqrt{\mathcal{M}_t}\Lambda_t^m\omega^{\mathrm{i}}(t)\|_{L_t^{\infty}L^2}^2
      &  +K\nu^{-\f23}\|\langle\pa_x\rangle^{m_1}\sqrt{\mathcal{M}_t}\Lambda_t^m\th_{\neq}^{\mathrm{i}}(t)\|_{L_{t}^{\infty}L^2}^2\\
    &+\int_{\nu^{-\f16}}^{\infty}\mathcal{CK}(\omega^{\mathrm{i}})(s)+K\nu^{-\f23}\mathcal{CK}(\theta^{\mathrm{i}})(s)ds\leq C\epsilon_0\nu^{\f23}.
\end{align*}
The following remark gives the estimate of the forcing term from the nonlinear interaction of $u^{\mathrm{i}}_{\neq}$ and $\theta^{\mathrm{i}}$. 
\begin{remark}\label{force}
Let $(u^{\mathrm{i}}, \omega^{\mathrm{i}}, \theta^{\mathrm{i}})$ be solutions to \eqref{eq:i}. It holds that
    \begin{align*}
  &\|\langle \partial_x\rangle\Lambda_t^n(u_{\neq}^{\mathrm{i}}\cdot\nabla \th^{\mathrm{i}})\|_{L^2}
  \approx \|\langle \partial_x\rangle\Lambda_t^n(\nabla^{\bot}\Delta^{-1}\omega_{\neq}^{\mathrm{i}}\cdot\nabla \th^{\mathrm{i}})\|_{L^2}\\
  =&\|\langle \partial_x\rangle\Lambda_t^n(\nabla_{L}^{\bot}\Delta^{-1}\omega_{\neq}^{\mathrm{i}}\cdot\nabla_L \th^{\mathrm{i}})\|_{L^2}\\
  \lesssim& \|\langle \partial_x\rangle\Lambda_t^{n+1}\Delta^{-1}\omega_{\neq}^{\mathrm{i}}\|_{L^2}\|\Lambda_t^{b+1}\theta^{\mathrm{i}}\|_{L^2}
  +\|\Lambda_t^{b+1}\Delta^{-1}\omega_{\neq}^{\mathrm{i}}\|_{L^2}
  \|\langle \partial_x\rangle\Lambda_t^{n+1}\theta^{\mathrm{i}}\|_{L^2}\\
  \lesssim &\frac{1}{t^2+1}\|\Lambda_t^{n+3}\omega_{\neq}^{\mathrm{i}}\|_{L^2}\|\langle \partial_x\rangle\Lambda_t^{n+1}\theta^{\mathrm{i}}\|_{L^2}\lesssim \frac{\epsilon_0^2\nu}{t^2+1},
\end{align*}
for any $m-3\geq n>b>2$, where $\nabla_L=(\partial_x, \partial_y+t\partial_x)$. 
\end{remark}

 \section{Estimates of $\omega^{\mathrm{e}}$ and $\th^{\mathrm{e}}$}
 In this section, we estimate the errors, namely, $(\omega^{\mathrm{e}}, \th^{\mathrm{e}})$. For the short-time region, we have the following proposition. 
\begin{proposition}
For $m-3>n>2$, there exists $\delta>0$ small and independent of $\nu$ such that if $0<\nu\leq1$, $\|\Lambda_t^m\omega^{\mathrm{i}}\|_{L^2}\leq \delta\nu^{\f13}$ and $\|\langle\partial_x\rangle\Lambda_t^m\th^{\mathrm{i}}\|_{L^2}\leq \delta\nu^{\f23}$, for all $t\in [0,\nu^{-\f16}]$, then it holds for $0\leq t\leq T_0=\nu^{-\f16}$ that
\begin{align*}
    &\|\Lambda_t^n\omega^{\mathrm{e}}\|_{L^2}\leq \delta\nu^{\frac{2}{3}},\\
    &\nu^{\frac{1}{6}}\|\partial_x\Lambda_t^n\theta^{\mathrm{e}}(t) \|_{L^2}+\| \Lambda_t^n\theta^{\mathrm{e}}(t) \|_{L^2}\leq \delta \nu.
\end{align*}
\end{proposition}
\begin{proof}
Applying $\Lambda_t^n$ to $\eqref{eq:e}_1$, since $\Lambda_t^n$ commutes with $\pa_t+y\pa_x$, after taking inner product with $\Lambda_t^n\omega^{\mathrm{e}}$, we have
\begin{align*}
    \f{d}{dt}\|\Lambda_t^n\omega^{\mathrm{e}}\|_{L^2}^2+2\nu\|\nabla\Lambda_t^n\omega^{\mathrm{e}}\|_{L^2}^2&+2\langle \Lambda_t^n(u^{\mathrm{i}}\cdot\nabla\omega^{\mathrm{e}}),\Lambda_t^n\omega^{\mathrm{e}}\rangle+2\langle \Lambda_t^n(u^{\mathrm{e}}\cdot\nabla\omega^{\mathrm{i}}),\Lambda_t^n\omega^{\mathrm{e}}\rangle\\
    &+2\langle \Lambda_t^n(u^{\mathrm{e}}\cdot\nabla\omega^{\mathrm{e}}),\Lambda_t^n\omega^{\mathrm{e}}\rangle=2\langle \Lambda_t^n\pa_x\th^{\mathrm{e}},\Lambda_t^n\omega^{\mathrm{e}}\rangle.
\end{align*}
By taking $w=\omega^{\mathrm{i}}$ and $f=\omega^{\mathrm{e}}$ in Lemma \ref{Lem: transport-shorttime}, we have
\begin{align*}
    2\left|\langle \Lambda_t^n(u^{\mathrm{i}}\cdot\nabla\omega^{\mathrm{e}}),\Lambda_t^n\omega^{\mathrm{e}}\rangle\right|\leq C(1+t)\|\Lambda_t^n\omega^{\mathrm{i}}\|_{L^2}\|\Lambda_t^n\omega^{\mathrm{e}}\|_{L^2}^2.
\end{align*}
By Plancherel's formula, we have
\begin{align*}
    2\left|\langle \Lambda_t^n(u^{\mathrm{e}}\cdot\nabla\omega^{\mathrm{i}}),\Lambda_t^n\omega^{\mathrm{e}}\rangle\right|\leq &C(1+t)\Big(\|\hat{u}^{\mathrm{e}}\|_{L^1}\|\Lambda_t^{n+1}\omega^{\mathrm{i}}\|_{L^2}\|\Lambda_t^n\omega^{\mathrm{e}}\|_{L^2}+\|\Lambda_t^nu^{\mathrm{e}}\|_{L^2}\|\Lambda_t\hat{\omega}^{\mathrm{i}}\|_{L^1}\|\Lambda_t^n\omega^{\mathrm{e}}\|_{L^2}\Big)\\
    \leq& C(1+t)\|\Lambda_t^{n+1}\omega^{\mathrm{i}}\|_{L^2}\|\Lambda_t^{n}\omega^{\mathrm{e}}\|_{L^2}^2.
\end{align*}
By taking $w=\omega^{\mathrm{e}}$ and $f=\omega^{\mathrm{e}}$ in Lemma \ref{Lem: transport-shorttime}, we have
\begin{align*}
    2\left|\langle \Lambda_t^n(u^{\mathrm{e}}\cdot\nabla\omega^{\mathrm{e}}),\Lambda_t^n\omega^{\mathrm{e}}\rangle\right|\leq C(1+t)\|\Lambda_t^n\omega^{\mathrm{e}}\|_{L^2}^3.
\end{align*}
Summing up, we know
\begin{equation}\label{eq:w^e-short}
\begin{aligned}
    &\f{d}{dt}\|\Lambda_t^n\omega^{\mathrm{e}}\|_{L^2}^2+2\nu\|\nabla\Lambda_t^n\omega^{\mathrm{e}}\|_{L^2}^2
    \\
    \leq& C(1+t)\big(\|\Lambda_t^n\omega^{\mathrm{e}}\|_{L^2}+\|\Lambda_t^{n+1}\omega^{\mathrm{i}}\|_{L^2}\big)\|\Lambda_t^n\omega^{\mathrm{e}}\|_{L^2}^2+C\|\Lambda_t^n\pa_x\th^{\mathrm{e}}\|_{L^2}\|\Lambda_t^n\omega^{\mathrm{e}}\|_{L^2}.
\end{aligned}
\end{equation}
Next, applying operator $\Lambda_t^n$ to $\eqref{eq:e}_2$, since $\Lambda_t^n$ commutes with $\pa_t+y\pa_x$, then after taking inner product with $\Lambda_t^n\th^{\mathrm{e}}$, we have
\begin{align*}
    &\f{d}{dt}\|\Lambda_t^n\th^{\mathrm{e}}\|_{L^2}^2+2\nu\|\nabla \Lambda_t^n\th^{\mathrm{e}}\|_{L^2}^2+2\langle \Lambda_t^n\big(u_{\neq}^{\mathrm{i}}\cdot\nabla\th^{\mathrm{i}}\big),\Lambda_t^n\th^{\mathrm{e}}\rangle +2\langle \Lambda_t^n\big(u^{\mathrm{i}}\cdot\nabla\th^{\mathrm{e}}\big),\Lambda_t^n\th^{\mathrm{e}}\rangle\\
    &+2\langle \Lambda_t^n\big(u^{\mathrm{e}}\cdot\nabla\th^{\mathrm{e}}\big),\Lambda_t^n\th^{\mathrm{e}}\rangle+2\langle \Lambda_t^n\big(u^{\mathrm{e}}\cdot\nabla\th^{\mathrm{i}}\big),\Lambda_t^n\th^{\mathrm{e}}\rangle=0.
\end{align*}
By Remark \ref{force}, we have
\begin{align*}
    \left|2\langle \Lambda_t^n\big(u_{\neq}^{\mathrm{i}}\cdot\nabla\th^{\mathrm{i}}\big),\Lambda_t^n\th^{\mathrm{e}}\rangle\right|\leq C\f{\ep_0^2\nu}{1+t^2}\|\Lambda_t^n\th^{\mathrm{e}}\|_{L^2}.
\end{align*}
And by taking $w=\omega^{\mathrm{i}}$, $f=\th^{\mathrm{e}}$ and $w=\omega^{\mathrm{e}}$, $f=\th^{\mathrm{e}}$ respectively in Lemma \ref{Lem: transport-shorttime}, we have 
\begin{align*}
    &\left|2\langle \Lambda_t^n\big(u^{\mathrm{i}}\cdot\nabla\th^{\mathrm{e}}\big),\Lambda_t^n\th^{\mathrm{e}}\rangle\right|+\left|2\langle \Lambda_t^n\big(u^{\mathrm{e}}\cdot\nabla\th^{\mathrm{e}}\big),\Lambda_t^n\th^{\mathrm{e}}\rangle\right|\\
    \leq& C(1+t)\big(\|\Lambda_t^n\omega^{\mathrm{i}}\|_{L^2}+\|\Lambda_t^n\omega^{\mathrm{e}}\|_{L^2}\big)\|\Lambda_t^n\th^{\mathrm{e}}\|_{L^2}^2.
\end{align*}
And by Plancherel's formula, we have
\begin{align*}
    \left|2\langle \Lambda_t^n\big(u^{\mathrm{e}}\cdot\nabla\th^{\mathrm{i}}\big),\Lambda_t^n\th^{\mathrm{e}}\rangle\right|\leq C(1+t)\|\Lambda_t^{n+1}\th^{\mathrm{i}}\|_{L^2}\|\Lambda_t^n\th^{\mathrm{e}}\|_{L^2}\|\Lambda_t^n\omega^{\mathrm{e}}\|_{L^2}.
\end{align*}
Summing up, we have
\begin{equation}\label{eq:theat^e}
    \begin{aligned}
        \f{d}{dt}\|\Lambda_t^n\th^{\mathrm{e}}\|_{L^2}^2&+2\nu\|\nabla \Lambda_t^n\th^{\mathrm{e}}\|_{L^2}^2
        \leq C(1+t)\big(\|\Lambda_t^n\omega^{\mathrm{i}}\|_{L^2}+\|\Lambda_t^n\omega^{\mathrm{e}}\|_{L^2}\big)\|\Lambda_t^n\th^{\mathrm{e}}\|_{L^2}^2\\
        &+C\f{\ep_0^2\nu}{1+t^2}\|\Lambda_t^n\th^{\mathrm{e}}\|_{L^2}+C(1+t)\|\Lambda_t^{n+1}\th^{\mathrm{i}}\|_{L^2}\|\Lambda_t^n\th^{\mathrm{e}}\|_{L^2}\|\Lambda_t^n\omega^{\mathrm{e}}\|_{L^2}.
    \end{aligned}
\end{equation}
Applying operator $\pa_x\Lambda_t^n$ to $\eqref{eq:e}_2$, since $\pa_x\Lambda_t^n$ commutes with $\pa_t+y\pa_x$, then after taking inner product with $\pa_x\Lambda_t^n\th^{\mathrm{e}}$, we have
\begin{align*}
    &\f{d}{dt}\|\pa_x\Lambda_t^n\th^{\mathrm{e}}\|_{L^2}^2+2\nu\|\nabla \pa_x\Lambda_t^n\th^{\mathrm{e}}\|_{L^2}^2+2\langle \pa_x\Lambda_t^n\big(u_{\neq}^{\mathrm{i}}\cdot\nabla\th^{\mathrm{i}}\big),\pa_x\Lambda_t^n\th^{\mathrm{e}}\rangle +2\langle \pa_x\Lambda_t^n\big(u^{\mathrm{i}}\cdot\nabla\th^{\mathrm{e}}\big),\pa_x\Lambda_t^n\th^{\mathrm{e}}\rangle\\
    &+2\langle \pa_x\Lambda_t^n\big(u^{\mathrm{e}}\cdot\nabla\th^{\mathrm{e}}\big),\pa_x\Lambda_t^n\th^{\mathrm{e}}\rangle+2\langle \pa_x\Lambda_t^n\big(u^{\mathrm{e}}\cdot\nabla\th^{\mathrm{i}}\big),\pa_x\Lambda_t^n\th^{\mathrm{e}}\rangle=0.
\end{align*}
By Remark \ref{force}, we have
\begin{align*}
    \left|2\langle \pa_x\Lambda_t^n\big(u_{\neq}^{\mathrm{i}}\cdot\nabla\th^{\mathrm{i}}\big),\pa_x\Lambda_t^n\th^{\mathrm{e}}\rangle\right|\leq C\f{\ep_0^2\nu}{t^2+1}\|\pa_x\Lambda_t^n\th^{\mathrm{e}}\|_{L^2}.
\end{align*}
    Next for the transport term, we have
    \begin{equation}\label{eq:u^i-theta^e}
    \begin{aligned}
        |\langle \partial_x\Lambda_t^n(u^{\mathrm{i}}\cdot\nabla \theta^{\mathrm{e}}), \partial_x\Lambda_t^n\theta^{\mathrm{e}} \rangle|
        \leq|\langle \Lambda_t^n(u^{\mathrm{i}}\cdot\nabla \partial_x\theta^{\mathrm{e}}), \partial_x\Lambda_t^n\theta^{\mathrm{e}} \rangle|
        +|\langle \Lambda_t^n(\partial_xu^{\mathrm{i}}\cdot\nabla \theta^{\mathrm{e}}), \partial_x\Lambda_t^n\theta^{\mathrm{e}} \rangle|.
    \end{aligned}
    \end{equation}
    The second term is the key term, 
    \begin{align*}
        |\langle \Lambda_t^n(\partial_xu^{\mathrm{i}}\cdot\nabla \theta^{\mathrm{e}}), \partial_x\Lambda_t^n\theta^{\mathrm{e}} \rangle|
        &\leq  C\|\Lambda_t^{n+1}u^{\mathrm{i}}\|_{L^2}(t+1)\|\Lambda_t^n\theta^{\mathrm{e}}\|_{L^2}\|\partial_x\Lambda_t^n\theta^{\mathrm{e}}\|_{L^2}\\
        &\quad+C\nu^{-\f13}\|\Lambda_t^{n+1}u^{\mathrm{i}}\|_{L^2}\|\partial_x\Lambda_t^n\theta^{\mathrm{e}}\|_{L^2}\big(\nu^{-\f16}\nu^{\f12}\|\nabla\Lambda_t^n\theta^{\mathrm{e}}\|_{L^2}\big)\\
        &\leq C\delta^3\nu^{\f13}\nu \nu^{\f56}+C\f{\delta^4\nu^{\f53}}{t^2}+\f{1}{100}\nu \nu^{-\f13}\|\nabla\Lambda_t^n\theta^{\mathrm{e}}\|_{L^2}^2.
    \end{align*}
    For the first term in \eqref{eq:u^i-theta^e}, by taking $w=\omega^{\mathrm{i}}$ and $f=\pa_x\th^{\mathrm{e}}$ in Lemma \ref{Lem: transport-shorttime}, we have
    \begin{align*}
        |\langle \Lambda_t^n(u^{\mathrm{i}}\cdot\nabla \partial_x\theta^{\mathrm{e}}), \partial_x\Lambda_t^n\theta^{\mathrm{e}} \rangle|\leq C(1+t)\|\Lambda_t^n\omega^{\mathrm{i}}\|\Lambda_t^n\pa_x\theta^{\mathrm{e}}\|_{L^2}^2.
    \end{align*}
    Similarly, by taking $w=\omega^{\mathrm{e}}$ and $f=\pa_x\th^{\mathrm{e}}$ in Lemma \ref{Lem: transport-shorttime}, we have
\begin{align*}
    &\left|\langle \pa_x\Lambda_t^n\big(u^{\mathrm{e}}\cdot\nabla\th^{\mathrm{e}}\big),\pa_x\Lambda_t^n\th^{\mathrm{e}}\rangle\right|\leq \left|\langle \Lambda_t^n\big(\pa_xu^{\mathrm{e}}\cdot\nabla\th^{\mathrm{e}}\big),\pa_x\Lambda_t^n\th^{\mathrm{e}}\rangle\right|+\left|\langle \Lambda_t^n\big(u^{\mathrm{e}}\cdot\nabla\pa_x\th^{\mathrm{e}}\big),\pa_x\Lambda_t^n\th^{\mathrm{e}}\rangle\right|\\
    \leq&C\big(\|\widehat{\pa_xu^{\mathrm{e}}}\|_{L^1}\|\Lambda_t^n\nabla\th^{\mathrm{e}}\|_{L^2}+\|\Lambda_t^n\pa_xu^{\mathrm{e}}\|_{L^2}\|\widehat{\nabla\th^{\mathrm{e}}}\|_{L^1}\big)\|\pa_x\Lambda_t^n\th^{\mathrm{e}}\|_{L^2}+C(1+t)\|\Lambda_t^n\omega^{\mathrm{e}}\|\Lambda_t^n\pa_x\theta^{\mathrm{e}}\|_{L^2}^2\\
    \leq&C\|\Lambda_t^n\omega^{\mathrm{e}}\|_{L^2}\|\Lambda_t^n\nabla\th^{\mathrm{e}}\|_{L^2}\|\pa_x\Lambda_t^n\th^{\mathrm{e}}\|_{L^2}+C(1+t)\|\Lambda_t^n\omega^{\mathrm{e}}\|\Lambda_t^n\pa_x\theta^{\mathrm{e}}\|_{L^2}^2\\
    \leq&C\nu^{-\f23}\|\Lambda_t^n\omega^{\mathrm{e}}\|_{L^2}^2\|\pa_x\Lambda_t^n\th^{\mathrm{e}}\|_{L^2}^2+\frac{1}{100}\nu^{-\f13}\nu\|\Lambda_t^n\nabla\th^{\mathrm{e}}\|_{L^2}^2+C(1+t)\|\Lambda_t^n\omega^{\mathrm{e}}\|\Lambda_t^n\pa_x\theta^{\mathrm{e}}\|_{L^2}^2.
\end{align*}
Finally, by Plancherel's formula and Young's convolution inequality, we have
\begin{align*}
    &\left|\langle \pa_x\Lambda_t^n\big(u^{\mathrm{e}}\cdot\nabla\th^{\mathrm{i}}\big),\pa_x\Lambda_t^n\th^{\mathrm{e}}\rangle\right|\leq \left|\langle \Lambda_t^n\big(\pa_xu^{\mathrm{e}}\cdot\nabla\th^{\mathrm{i}}\big),\pa_x\Lambda_t^n\th^{\mathrm{e}}\rangle\right|+\left|\langle \Lambda_t^n\big(u^{\mathrm{e}}\cdot\pa_x\nabla\th^{\mathrm{i}}\big),\pa_x\Lambda_t^n\th^{\mathrm{e}}\rangle\right|\\
    \leq& C\big(\|\pa_x\Lambda_t^nu^{\mathrm{e}}\|_{L^2}\|\widehat{\nabla\th^{\mathrm{i}}}\|_{L^1}+\|\widehat{\pa_xu^{\mathrm{e}}}\|_{L^1}\|\Lambda_t^n\nabla\th^{\mathrm{i}}\|_{L^2}\big)\|\pa_x\Lambda_t^n\th^{\mathrm{e}}\|_{L^2}\\
    &+C\big(\|\Lambda_t^nu^{\mathrm{e}}\|_{L^2}\|\widehat{\pa_x\nabla\th^{\mathrm{i}}}\|_{L^1}+\|\widehat{u^{\mathrm{e}}}\|_{L^1}\|\pa_x\Lambda_t^n\nabla\th^{\mathrm{i}}\|_{L^2}\big)\|\pa_x\Lambda_t^n\th^{\mathrm{e}}\|_{L^2}\\
    \leq& C\big(\|\Lambda_t^n\omega^{\mathrm{e}}\|_{L^2}\|\Lambda_t^n\nabla\th^{\mathrm{i}}\|_{L^2}+\|\Lambda_t^n\omega^{\mathrm{e}}\|_{L^2}\|\pa_x\Lambda_t^n\nabla\th^{\mathrm{i}}\|_{L^2}\big)\|\pa_x\Lambda_t^n\th^{\mathrm{e}}\|_{L^2}\\
    \leq & (1+t)\nu^{\f23}\|\Lambda_t^n\omega^{\mathrm{e}}\|_{L^2}\|\pa_x\Lambda_t^n\th^{\mathrm{e}}\|_{L^2}.
\end{align*}
Then summing up and combined with \eqref{eq:w^e-short}, \eqref{eq:theat^e}, we obtain the result. 
\end{proof}
Next, we consider the long-time region $t\geq \nu^{-\f16}$ and prove the following proposition. 
\begin{proposition}
    For $2<n<m-3$, there exists $\delta>0$ such that for $0<\nu\leq1$, if for $t\geq \nu^{-\f16}$, $\|\sqrt{\mathcal{M}_t}\Lambda_t^n\omega^{\mathrm{e}}\|_{L^2}\leq\delta\nu^{\f12}$ and $\|\langle\pa_x\rangle^{\f13}\sqrt{\mathcal{M}_t}\Lambda_t^n\th^{\mathrm{e}}\|_{L^2}\leq\delta\nu^{\f56}$ then it holds for $t\geq \nu^{-\frac{1}{6}}$ that
    \begin{align*}
        &\f{d}{dt}
        \Big(\|\Lambda_t^n\sqrt{\mathcal{M}_t}\omega^{\mathrm{e}}(t)\|_{L^2}^2+ K\nu^{-\f23}\|\langle \partial_x\rangle^{\frac{1}{3}} \sqrt{\mathcal{M}_t}\Lambda_t^n\theta^{\mathrm{e}}(t) \|_{L^2}^2\Big)
        +2\Big(\mathcal{CK}(\omega^{\mathrm{e}})
       +K\nu^{-\f23}\mathcal{CK}(\theta^{\mathrm{e}})\Big)\\
       &\leq \nu^{\f13}\mathcal{CK}(\omega^{\mathrm{i}})+K\nu^{-\f13}\mathcal{CK}(\theta^{\mathrm{i}})+\f{C\epsilon_0^2\delta\nu^{\f76}}{1+t^2},
    \end{align*}
    where
    \begin{align*}
        \mathcal{CK}(\omega^{\mathrm{e}})=\frac{\delta_0}{4}\Big(\nu\left\|(D_x,D_y)\sqrt{\mathcal{M}_t}\Lambda_t^n\omega^{\mathrm{e}}\right\|_{L^2}^2&+\nu^{\f13}\left\||D_x|^{\frac{1}{3}}\sqrt{\mathcal{M}_t}\Lambda_t^n\omega^{\mathrm{e}}\right\|_{L^2}^2\\
    &+\left\|\nabla\sqrt{\mathcal{M}_t}\Lambda_t^nu_{\neq}^{\mathrm{e},2}\right\|_{L^2}^2+\left\|\sqrt{\mathcal{M}_t\Upsilon_t}\Lambda_t^n\omega^{\mathrm{e}}\right\|_{L^2}^2\Big),
    \end{align*}
    and
    \begin{align*}
        \mathcal{CK}(\th^{\mathrm{e}})=&\frac{\delta_0}{4}\Big(\nu\left\|(D_x,D_y)\langle\pa_x\rangle^{\f13}\sqrt{\mathcal{M}_t}\Lambda_t^n\th^{\mathrm{e}}\right\|_{L^2}^2\\
        &\quad \quad\quad \quad\quad \quad
        +\|\langle\pa_x\rangle^{\f13}\sqrt{\Upsilon_t\mathcal{M}_t}\Lambda_t^n\th^{\mathrm{e}}\|_{L^2}^2+\nu^{\f13}\left\||D_x|^{\frac{1}{3}}\langle\pa_x\rangle^{\f13}\sqrt{\mathcal{M}_t}\Lambda_t^n\th^{\mathrm{e}}\right\|_{L^2}^2\Big).
    \end{align*}
\end{proposition}
\begin{proof}
Applying $\Lambda_t^n$ to $\eqref{eq:e}_1$, and taking inner product with $\mathcal{M}_t\Lambda_t^n\omega^{\mathrm{e}}$, we have
\begin{align*}
    &\f{d}{dt}\|\sqrt{\mathcal{M}_t}\Lambda_t^n\omega^{\mathrm{e}}\|_{L^2}^2+2\nu\|(D_x,D_y)\sqrt{\mathcal{M}_t}\Lambda_t^n\omega^{\mathrm{e}}\|_{L^2}^2+\langle (k\pa_{\xi}-\pa_t)\mathcal{M}(t,D)\Lambda_t^n\omega^{\mathrm
    {e}},\Lambda_t^n\omega^{\mathrm{e}}\rangle\\
    &+2\langle \Lambda_t^n(u^{\mathrm{i}}\cdot\nabla\omega^{\mathrm{e}}),\mathcal{M}_t\Lambda_t^n\omega^{\mathrm{e}}\rangle+2\langle \Lambda_t^n(u^{\mathrm{e}}\cdot\nabla\omega^{\mathrm{i}}),\mathcal{M}_t\Lambda_t^n\omega^{\mathrm{e}}\rangle+2\langle \Lambda_t^n(u^{\mathrm{e}}\cdot\nabla\omega^{\mathrm{e}}),\mathcal{M}_t\Lambda_t^n\omega^{\mathrm{e}}\rangle\\
    =&2\langle \Lambda_t^n\pa_x\th^{\mathrm{e}},\mathcal{M}_t\Lambda_t^n\omega^{\mathrm{e}}\rangle.
\end{align*}
Then by taking $f=\Lambda_t^n\omega^{\mathrm{e}}$ in \eqref{energy-dissipation} and using the fact that $u_{\neq}^{\mathrm{e},2}=\pa_x(-\Delta)^{-1}\omega_{\neq}^{\mathrm{e}}$, we have
\begin{align*}
    &\f{d}{dt}\|\sqrt{\mathcal{M}_t}\Lambda_t^n\omega^{\mathrm{e}}\|_{L^2}^2+\frac{\delta_0}{2}\Big(\nu\left\|(D_x,D_y)\sqrt{\mathcal{M}_t}\Lambda_t^n\omega^{\mathrm{e}}\right\|_{L^2}^2+\nu^{\f13}\left\||D_x|^{\frac{1}{3}}\sqrt{\mathcal{M}_t}\Lambda_t^n\omega^{\mathrm{e}}\right\|_{L^2}^2\\
    &+\left\|\nabla\sqrt{\mathcal{M}_t}\Lambda_t^nu_{\neq}^{\mathrm{e},2}\right\|_{L^2}^2+\left\|\sqrt{\mathcal{M}_t\Upsilon_t}\Lambda_t^n\omega^{\mathrm{e}}\right\|_{L^2}^2\Big)\\
    \leq&-2\langle \Lambda_t^n(u^{\mathrm{i}}\cdot\nabla\omega^{\mathrm{e}}),\mathcal{M}_t\Lambda_t^n\omega^{\mathrm{e}}\rangle-2\langle \Lambda_t^n(u^{\mathrm{e}}\cdot\nabla\omega^{\mathrm{i}}),\mathcal{M}_t\Lambda_t^n\omega^{\mathrm{e}}\rangle-2\langle \Lambda_t^n(u^{\mathrm{e}}\cdot\nabla\omega^{\mathrm{e}}),\mathcal{M}_t\Lambda_t^n\omega^{\mathrm{e}}\rangle\\
    &+2\langle \Lambda_t^n\pa_x\th^{\mathrm{e}},\mathcal{M}_t\Lambda_t^n\omega^{\mathrm{e}}\rangle=:-I_1-I_2-I_3+I_4.
\end{align*}
Since $\mathrm{div}u^{\mathrm{i}}=\mathrm{div}u^{\mathrm{e}}=0$, we know
\begin{align*}
    \langle u^{\mathrm{i}}\cdot\nabla(\sqrt{\mathcal{M}_t}\Lambda_t^n\omega^{\mathrm{e}}),\sqrt{\mathcal{M}_t}\Lambda_t^n\omega^{\mathrm{e}}\rangle=\langle u^{\mathrm{e}}\cdot\nabla(\sqrt{\mathcal{M}_t}\Lambda_t^n\omega^{\mathrm{e}}),\sqrt{\mathcal{M}_t}\Lambda_t^n\omega^{\mathrm{e}}\rangle=0,
\end{align*}
then for $I_1$ and $I_3$, we have
\begin{align*}
    |I_1|=2\left|\langle \sqrt{\mathcal{M}_t}\Lambda_t^n(u^{\mathrm{i}}\cdot\nabla\omega^{\mathrm{e}}),\sqrt{\mathcal{M}_t}\Lambda_t^n\omega^{\mathrm{e}}\rangle-\langle u^{\mathrm{i}}\cdot\nabla(\sqrt{\mathcal{M}_t}\Lambda_t^n\omega^{\mathrm{e}}),\sqrt{\mathcal{M}_t}\Lambda_t^n\omega^{\mathrm{e}}\rangle\right|,\\
    |I_3|=2\left|\langle \sqrt{\mathcal{M}_t}\Lambda_t^n(u^{\mathrm{e}}\cdot\nabla\omega^{\mathrm{e}}),\sqrt{\mathcal{M}_t}\Lambda_t^n\omega^{\mathrm{e}}\rangle-\langle u^{\mathrm{e}}\cdot\nabla(\sqrt{\mathcal{M}_t}\Lambda_t^n\omega^{\mathrm{e}}),\sqrt{\mathcal{M}_t}\Lambda_t^n\omega^{\mathrm{e}}\rangle\right|.
\end{align*}
And we can decompose $I_1=2(I_{11}+I_{12}+I_{13})$ as 
\begin{align*}
    I_{11}:=\big\langle \sqrt{\mathcal{M}_t}\Lambda_t^n(u_{\neq}^{\mathrm{i},2}\pa_y\omega^{\mathrm{e}})-u_{\neq}^{\mathrm{i},2}\pa_y\big(\sqrt{\mathcal{M}_t}\Lambda_t^n\omega^{\mathrm{e}}\big),\sqrt{\mathcal{M}_t}\Lambda_t^n\omega^{\mathrm{e}}\big\rangle,\\
         I_{12}:=\big\langle \sqrt{\mathcal{M}_t}\Lambda_t^n(u_{0}^{\mathrm{i},1}\pa_x\omega^{\mathrm{e}})-u_{0}^{\mathrm{i},1}\pa_x\big(\sqrt{\mathcal{M}_t}\Lambda_t^n\omega^{\mathrm{e}}\big),\sqrt{\mathcal{M}_t}\Lambda_t^n\omega^{\mathrm{e}}\big\rangle,\\
         I_{13}:=\big\langle \sqrt{\mathcal{M}_t}\Lambda_t^n(u_{\neq}^{\mathrm{i},1}\pa_x\omega^{\mathrm{e}})-u_{\neq}^{\mathrm{i},1}\pa_x\big(\sqrt{\mathcal{M}_t}\Lambda_t^n\omega^{\mathrm{e}}\big),\sqrt{\mathcal{M}_t}\Lambda_t^n\omega^{\mathrm{e}}\big\rangle.
\end{align*}
Then by taking $v=\omega^{\mathrm{i}}$ and $f=g=\omega^{\mathrm{e}}$ in Lemma \ref{Lem:I11}, \ref{Lem:u_0 comm-long}, and \ref{Lem:I31}, we have
\begin{align*}
    &|I_1|\leq C\|\nabla\sqrt{\mathcal{M}_t}\Lambda_t^nu_{\neq}^{\mathrm{i},2}\|_{L^2}\big(\nu^{-\f13}\|\sqrt{\mathcal{M}_t}\Lambda_t^n\omega^{\mathrm{e}}\|_{L^2}\|\sqrt{\Upsilon_t\mathcal{M}_t}\Lambda_t^n\omega^{\mathrm{e}}\|_{L^2}\\
    &+\nu^{\f16}\|D_y\Lambda_t^n\omega^{\mathrm{e}}\|_{L^2}\|\sqrt{\mathcal{M}_t}\Lambda_t^n\omega^{\mathrm{e}}\|_{L^2}\big)+C\|\Lambda_t^n\omega^{\mathrm{i}}\|_{L^2}\||D_x|^{\f13}\sqrt{\mathcal{M}_t}\Lambda_t^n\omega^{\mathrm{e}}\|_{L^2}\||D_x|^{\f13}\sqrt{\mathcal{M}_t}\Lambda_t^n\omega^{\mathrm{e}}\|_{L^2}\\
    &+C\|\sqrt{\mathcal{M}_t}\Lambda_t^n\omega^{\mathrm{i}}\|_{L^2}\big(\||D_x|^{\f13}\sqrt{\mathcal{M}_t}\Lambda_t^n\omega^{\mathrm{e}}\|_{L^2}^2+\nu^{\f23}\|D_y\Lambda_t^n\omega^{\mathrm{e}}\|_{L^2}\||D_x|^{\f13}\sqrt{\mathcal{M}_t}\Lambda_t^n\omega^{\mathrm{e}}\|_{L^2}\big).
\end{align*}
Similarly, we get
\begin{align*}
    &|I_3|\leq C\|\nabla\sqrt{\mathcal{M}_t}\Lambda_t^nu_{\neq}^{\mathrm{e},2}\|_{L^2}\big(\nu^{-\f13}\|\sqrt{\mathcal{M}_t}\Lambda_t^n\omega^{\mathrm{e}}\|_{L^2}\|\sqrt{\Upsilon_t\mathcal{M}_t}\Lambda_t^n\omega^{\mathrm{e}}\|_{L^2}\\
    &+\nu^{\f16}\|D_y\Lambda_t^n\omega^{\mathrm{e}}\|_{L^2}\|\sqrt{\mathcal{M}_t}\Lambda_t^n\omega^{\mathrm{e}}\|_{L^2}\big)+C\|\Lambda_t^n\omega^{\mathrm{e}}\|_{L^2}\||D_x|^{\f13}\sqrt{\mathcal{M}_t}\Lambda_t^n\omega^{\mathrm{e}}\|_{L^2}\||D_x|^{\f13}\sqrt{\mathcal{M}_t}\Lambda_t^n\omega^{\mathrm{e}}\|_{L^2}\\
    &+C\|\sqrt{\mathcal{M}_t}\Lambda_t^n\omega^{\mathrm{e}}\|_{L^2}\big(\||D_x|^{\f13}\sqrt{\mathcal{M}_t}\Lambda_t^n\omega^{\mathrm{e}}\|_{L^2}^2+\nu^{\f23}\|D_y\Lambda_t^n\omega^{\mathrm{e}}\|_{L^2}\||D_x|^{\f13}\sqrt{\mathcal{M}_t}\Lambda_t^n\omega^{\mathrm{e}}\|_{L^2}\big).
\end{align*}
Next, for $I_2$, we decompose it as $I_2=2(I_{21}+I_{22}+I_{23})$ with
\begin{align*}
    I_{21}:=\big\langle \sqrt{\mathcal{M}_t}\Lambda_t^n(u_{\neq}^{\mathrm{e},2}\pa_y\omega^{\mathrm{i}}),\sqrt{\mathcal{M}_t}\Lambda_t^n\omega^{\mathrm{e}}\big\rangle,\\
         I_{22}:=\big\langle \sqrt{\mathcal{M}_t}\Lambda_t^n(u_{0}^{\mathrm{e},1}\pa_x\omega^{\mathrm{i}}),\sqrt{\mathcal{M}_t}\Lambda_t^n\omega^{\mathrm{e}}\big\rangle,\\
         I_{23}:=\big\langle \sqrt{\mathcal{M}_t}\Lambda_t^n(u_{\neq}^{\mathrm{e},1}\pa_x\omega^{\mathrm{i}}),\sqrt{\mathcal{M}_t}\Lambda_t^n\omega^{\mathrm{e}}\big\rangle.
\end{align*}
For $I_{21}$, by taking $w=\omega^{\mathrm{e}}_{\neq}$, $f=\omega^{\mathrm{i}}$, and $g=\omega^{\mathrm{e}}$ in Lemma \ref{Lem:I11}, we have
\begin{align*}
    |I_{21}|\leq C\|\nabla\sqrt{\mathcal{M}_t}\Lambda_t^nu_{\neq}^{\mathrm{e},2}\|_{L^2}\big(\nu^{-\f13}\|\sqrt{\mathcal{M}_t}\Lambda_t^n\omega^{\mathrm{i}}\|_{L^2}\|\sqrt{\Upsilon_t\mathcal{M}_t}\Lambda_t^n\omega^{\mathrm{e}}\|_{L^2}\\
    +\nu^{\f16}\|D_y\Lambda_t^n\omega^{\mathrm{i}}\|_{L^2}\|\sqrt{\mathcal{M}_t}\Lambda_t^n\omega^{\mathrm{e}}\|_{L^2}\big).
\end{align*}
And by Plancherel's formula and Young's convolution inequality, we know
\begin{align*}
    |I_{22}|\leq C\|\sqrt{\mathcal{M}_t}\Lambda_t^n\pa_x\omega^{\mathrm{i}}\|_{L^2}\|\sqrt{\mathcal{M}_t}\Lambda_t^nu^{\mathrm{e},1}_0\|_{L^2}\|\sqrt{\mathcal{M}_t}\Lambda_t^n\omega^{\mathrm{e}}_{\neq}\|_{L^2},
\end{align*}
and
\begin{align*}
    |I_{23}|&\leq C\big(\|\widehat{u_{\neq}^{\mathrm{e},1}}\|_{L^1}\|\sqrt{\mathcal{M}_t}\Lambda_t^n\pa_x\omega^{\mathrm{i}}\|_{L^2}+C\|\widehat{\pa_x\omega^{\mathrm{i}}}\|_{L^1}\|\sqrt{\mathcal{M}_t}\Lambda_t^nu_{\neq}^{\mathrm{e},1}\|_{L^2}\big)\|\sqrt{\mathcal{M}_t}\Lambda_t^n\omega^{\mathrm{e}}\|_{L^2}\\
    &\leq C\|\sqrt{\mathcal{M}_t}\Lambda_t^n\pa_x\omega^{\mathrm{i}}\|_{L^2}\|\sqrt{\mathcal{M}_t}\Lambda_t^nu_{\neq}^{\mathrm{e},1}\|_{L^2}\|\sqrt{\mathcal{M}_t}\Lambda_t^n\omega^{\mathrm{e}}\|_{L^2}.
\end{align*}
For $I_4$, we have
\begin{align*}
    |I_4|\leq C\|\Lambda_t^n|D_x|^{\f13}\langle\pa_x\rangle^{\f13}\sqrt{\mathcal{M}_t}\th^{\mathrm{e}}\|_{L^2}\||D_x|^{\f13}\sqrt{\mathcal{M}_t}\Lambda_t^n\omega^{\mathrm{e}}\|_{L^2}.
\end{align*}
Then we can conclude that
\begin{equation}\label{ineq:w}
    \begin{aligned}
        &\f{d}{dt}\|\sqrt{\mathcal{M}_t}\Lambda_t^n\omega^{\mathrm{e}}\|_{L^2}^2+2\mathcal{CK}(\omega^{\mathrm{e}})\\
    \leq&C\big(\mathcal{CK}^{\f12}(\omega^{\mathrm{i}})+\mathcal{CK}^{\f12}(\omega^{\mathrm{e}})\big)\big(\nu^{-\f13}\delta\nu^{\f12}\mathcal{CK}^{\f12}(\omega^{\mathrm{e}})+\nu^{\f16}\nu^{-\f12}\mathcal{CK}^{\f12}(\omega^{\mathrm{e}})\delta\nu^{\f12}\big)\\
    &+C\big(C\ep_0\nu^{\f13}+\delta\nu^{\f23}\big)\nu^{-\f13}\mathcal{CK}(\omega^{\mathrm{e}})\\
    &+C\big(\delta\nu^{\f13}+\delta\nu^{\f12}\big)\big(\nu^{-\f13}\mathcal{CK}(\omega^{\mathrm{e}})+\nu^{\f23}\nu^{-\f12}\mathcal{CK}^{\f12}(\omega^{\mathrm{e}})\nu^{-\f16}\mathcal{CK}^{\f12}(\omega^{\mathrm{e}})\big)\\
    &+C\mathcal{CK}^{\f12}(\omega^{\mathrm{e}})\big(\nu^{-\f13}\delta\nu^{\f13}\mathcal{CK}^{\f12}(\omega^{\mathrm{e}})
    +\nu^{\f16}\nu^{-\f12}\mathcal{CK}^{\f12}(\omega^{\mathrm{i}})\delta\nu^{\f12}\big)+C\nu^{-\f16}\mathcal{CK}^{\f12}(\omega^{\mathrm{i}})\delta\nu^{\f12}\nu^{-\f16}\mathcal{CK}^{\f12}(\omega^{\mathrm{e}})\\
    &+C\nu^{-\f16}\mathcal{CK}^{\f12}(\omega^{\mathrm{i}})\nu^{-\f16}\mathcal{CK}^{\f12}(\omega^{\mathrm{e}})\delta\nu^{\f12}+C\nu^{-\f16}\mathcal{CK}^{\f12}(\th^{\mathrm{e}})\nu^{-\f16}\mathcal{CK}^{\f12}(\omega^{\mathrm{e}})\\
    \leq&C_{\delta_0}\delta\nu^{\f13}\mathcal{CK}(\omega^{\mathrm{i}})+C_{\delta_0}(\ep_0+\delta+K^{-1})\mathcal{CK}(\omega^{\mathrm{e}})+\f12K\nu^{-\f23}\mathcal{CK}(\th^{\mathrm{e}}).
    \end{aligned}
\end{equation}
Through a direct calculation, we know
\begin{equation}
    \begin{aligned}
        &\f{d}{dt}\|\langle\pa_x\rangle^{\f13}\sqrt{\mathcal{M}_t}\Lambda_t^n\th^{\mathrm{e}}\|_{L^2}^2+\frac{\delta_0}{2}\Big(\nu\left\|(D_x,D_y)\langle\pa_x\rangle^{\f13}\sqrt{\mathcal{M}_t}\Lambda_t^n\th^{\mathrm{e}}\right\|_{L^2}^2\\
        &\quad \quad\quad \quad\quad \quad
        +\|\langle\pa_x\rangle^{\f13}\sqrt{\Upsilon_t\mathcal{M}_t}\Lambda_t^n\th^{\mathrm{e}}\|_{L^2}^2+\nu^{\f13}\left\||D_x|^{\frac{1}{3}}\langle\pa_x\rangle^{\f13}\sqrt{\mathcal{M}_t}\Lambda_t^n\th^{\mathrm{e}}\right\|_{L^2}^2\Big)\\
    \leq&-2\langle \langle\pa_x\rangle^{\f13}\Lambda_t^n(u_{\neq}^{\mathrm{i}}\cdot\nabla\th^{\mathrm{i}}),\langle\pa_x\rangle^{\f13}\mathcal{M}_t\Lambda_t^n\th^{\mathrm{e}}\rangle
    -2\langle \langle\pa_x\rangle^{\f13}\Lambda_t^n(u^{\mathrm{i}}\cdot\nabla\th^{\mathrm{e}}),
    \langle\pa_x\rangle^{\f13}\mathcal{M}_t\Lambda_t^n\th^{\mathrm{e}}\rangle\\
    &-2\langle \langle\pa_x\rangle^{\f13}\Lambda_t^n(u^{\mathrm{e}}\cdot\nabla\th^{\mathrm{e}}),\langle\pa_x\rangle^{\f13}\mathcal{M}_t\Lambda_t^n\th^{\mathrm{e}}\rangle
    -2\langle \langle\pa_x\rangle^{\f13}\Lambda_t^n(u^{\mathrm{e}}\cdot\nabla\th^{\mathrm{i}}),
    \langle\pa_x\rangle^{\f13}\mathcal{M}_t\Lambda_t^n\th^{\mathrm{e}}\rangle\\
    =:&-J_1-J_2-J_3-J_4.
    \end{aligned}
\end{equation}
By Remark \ref{force}, we know
\begin{align*}
    |J_1|\leq C\f{\ep_0^2\nu}{1+t^2}\|\langle\pa_x\rangle^{\f13}\sqrt{\mathcal{M}_t}\Lambda_t^n\th^{\mathrm{e}}\|_{L^2}.
\end{align*}
For $J_2$, we decompose it as $J_2=2(J_{21}+J_{22})$ with
\begin{align*}
    J_{21}=\langle \langle\pa_x\rangle^{\f13}\sqrt{\mathcal{M}_t}\Lambda_t^n(u^{\mathrm{i},1}\pa_x\th^{\mathrm{e}}),
    \langle\pa_x\rangle^{\f13}\sqrt{\mathcal{M}_t}\Lambda_t^n\th^{\mathrm{e}}\rangle,\\
    J_{22}=\langle\langle\pa_x\rangle^{\f13}\sqrt{\mathcal{M}_t}\Lambda_t^n(u^{\mathrm{i},2}\pa_y\th^{\mathrm{e}}),
    \langle\pa_x\rangle^{\f13}\sqrt{\mathcal{M}_t}\Lambda_t^n\th^{\mathrm{e}}\rangle.
\end{align*}
For $J_{21}$, we have
\begin{align*}
    J_{21}=&\langle \langle\pa_x\rangle^{\f13}\sqrt{\mathcal{M}_t}\Lambda_t^n(u^{\mathrm{i},1}\pa_x\th^{\mathrm{e}})-\sqrt{\mathcal{M}_t}\Lambda_t^nu^{\mathrm{i},1}\pa_x\langle\pa_x\rangle^{\f13}\th^{\mathrm{e}},
    \langle\pa_x\rangle^{\f13}\sqrt{\mathcal{M}_t}\Lambda_t^n\th^{\mathrm{e}}\rangle\\
    &+\langle \sqrt{\mathcal{M}_t}\Lambda_t^nu^{\mathrm{i},1}\pa_x\langle\pa_x\rangle^{\f13}\th^{\mathrm{e}},
    \sqrt{\mathcal{M}_t}\Lambda_t^n\langle\pa_x\rangle^{\f13}\th^{\mathrm{e}}\rangle=:J_{211}+J_{212}.
\end{align*}
By using the fact that $\big|\langle k\rangle^{\f13}-\langle k-l\rangle^{\f13}\big||k-l|
\lesssim \big|\langle l\rangle^{\f13}+\langle k-l\rangle^{\f13}\big||l|$, and applying Plancherel's formula and Young's convolution inequality, we know
\begin{align*}
    |J_{211}|\leq&\sum_{k,l\in\mathbb Z,k\neq0}\iint\sqrt{\mathcal{M}_t(k,\xi)}\Lambda_t^n(k,\xi)\big|\langle k\rangle^{\f13}-\langle k-l\rangle^{\f13}\big||k-l||\hat{u}^{\mathrm{i},1}(l,\eta)||\hat{\th}^{\mathrm{e}}(k-l,\xi-\eta)|\\
    &\times \sqrt{\mathcal{M}_t(k,\xi)}\Lambda_t^n(k,\xi)\langle k\rangle^{\f13}\left|\overline{\hat{\th}^{\mathrm{e}}(k,\xi)}\right|d\xi d\eta\\
    \leq&\sum_{k,l\in\mathbb Z,k\neq0}\iint\sqrt{\mathcal{M}_t(k,\xi)}\Lambda_t^n(k,\xi)\big|\langle l\rangle^{\f13}+\langle k-l\rangle^{\f13}\big||l||\hat{u}^{\mathrm{i},1}(l,\eta)||\hat{\th}^{\mathrm{e}}(k-l,\xi-\eta)|\\
    &\times \sqrt{\mathcal{M}_t(k,\xi)}\Lambda_t^n(k,\xi)\langle k\rangle^{\f13}\left|\overline{\hat{\th}^{\mathrm{e}}(k,\xi)}\right|d\xi d\eta\\
    \leq&C\|\sqrt{\mathcal{M}_t}\Lambda_t^n\th^{\mathrm{e}}\|_{L^2}\|\sqrt{\mathcal{M}_t}\Lambda_t^n|D_x|^{\f13}\pa_xu^{\mathrm{i},1}\|_{L^2}\|\sqrt{\mathcal{M}_t}\Lambda_t^n|D_x|^{\f13}\th^{\mathrm{e}}\|_{L^2}\\
    &+C\|\sqrt{\mathcal{M}_t}\Lambda_t^nD_xu^{\mathrm{i},1}\|_{L^2}\|\sqrt{\mathcal{M}_t}\Lambda_t^n|D_x|^{\f13}\th^{\mathrm{e}}\|_{L^2}^2.
\end{align*}
By taking $w=\omega^{\mathrm{i}}$ and $f=\langle \pa_x\rangle^{\f13}\th^{\mathrm{e}}$ in Lemma \ref{Lem:u_0 comm-long} and \ref{Lem:I31}, we have
\begin{align*}
    |J_{212}|\leq& C\|\sqrt{\mathcal{M}_t}\Lambda_t^n\omega^{\mathrm{i}}\|_{L^2}\big(\||D_x|^{\f13}\sqrt{\mathcal{M}_t}\Lambda_t^n\langle \pa_x\rangle^{\f13}\th^{\mathrm{e}}\|_{L^2}^2\\
    &+\nu^{\f23}\|D_y\Lambda_t^n\langle \pa_x\rangle^{\f13}\th^{\mathrm{e}}\|_{L^2}\||D_x|^{\f13}\sqrt{\mathcal{M}_t}\Lambda_t^n\langle \pa_x\rangle^{\f13}\th^{\mathrm{e}}\|_{L^2}\big).
\end{align*}
Next for $J_{22}$, we have
\begin{align*}
    |J_{22}|\leq &\|\langle\pa_x\rangle^{\f13}\sqrt{\mathcal{M}_t}\Lambda_t^n(u^{\mathrm{i},2}\pa_y\th^{\mathrm{e}})\|_{L^2}\|\langle\pa_x\rangle^{\f13}\sqrt{\mathcal{M}_t}\Lambda_t^n\th^{\mathrm{e}}\|_{L^2}\\
    \leq&C\Big(\|\widehat{\pa_y\th^{\mathrm{e}}}\|_{L^1}\||D_x|^{\f13}\sqrt{\mathcal{M}_t}\Lambda_t^nu^{\mathrm{i},2}\|_{L^2}+C\|\hat{u}^{\mathrm{i},2}\|_{L^1}\|D_x|^{\f13}\sqrt{\mathcal{M}_t}\Lambda_t^nD_y\th^{\mathrm{e}}\|_{L^2}\Big)\|\langle\pa_x\rangle^{\f13}\sqrt{\mathcal{M}_t}\Lambda_t^n\th^{\mathrm{e}}\|_{L^2}\\
    \leq&C\||D_x|^{\f13}\sqrt{\mathcal{M}_t}\Lambda_t^nu^{\mathrm{i},2}\|_{L^2}\|D_x|^{\f13}\sqrt{\mathcal{M}_t}\Lambda_t^nD_y\th^{\mathrm{e}}\|_{L^2}\|\langle\pa_x\rangle^{\f13}\sqrt{\mathcal{M}_t}\Lambda_t^n\th^{\mathrm{e}}\|_{L^2}\\
    \leq&C\f{\ep_0^2\nu^{\f13}}{1+t^2}\|D_x|^{\f13}\sqrt{\mathcal{M}_t}\Lambda_t^nD_y\th^{\mathrm{e}}\|_{L^2}\|\langle\pa_x\rangle^{\f13}\sqrt{\mathcal{M}_t}\Lambda_t^n\th^{\mathrm{e}}\|_{L^2}\\
    \leq & C\f{\ep_0^2\nu^{\f12}}{1+t}\|D_x|^{\f13}\sqrt{\mathcal{M}_t}\Lambda_t^nD_y\th^{\mathrm{e}}\|_{L^2}\|\langle\pa_x\rangle^{\f13}\sqrt{\mathcal{M}_t}\Lambda_t^n\th^{\mathrm{e}}\|_{L^2}.
\end{align*}
For $J_3$, we decompose it as $J_3=2(J_{31}+J_{32}+J_{33})$, with 
\begin{align*}
    J_{31}=\langle \langle\pa_x\rangle^{\f13}\sqrt{\mathcal{M}_t}\Lambda_t^n(u^{\mathrm{e},1}_{\neq}\pa_x\th^{\mathrm{e}}),\langle\pa_x\rangle^{\f13}\sqrt{\mathcal{M}_t}\Lambda_t^n\th^{\mathrm{e}}\rangle,\\
    J_{32}=\langle \langle\pa_x\rangle^{\f13}\sqrt{\mathcal{M}_t}\Lambda_t^n(u^{\mathrm{e},1}_0\pa_x\th^{\mathrm{e}}),\langle\pa_x\rangle^{\f13}\sqrt{\mathcal{M}_t}\Lambda_t^n\th^{\mathrm{e}}\rangle,\\
    J_{33}=\langle \langle\pa_x\rangle^{\f13}\sqrt{\mathcal{M}_t}\Lambda_t^n(u^{\mathrm{e},2}_{\neq}\pa_y\th^{\mathrm{e}}),\langle\pa_x\rangle^{\f13}\sqrt{\mathcal{M}_t}\Lambda_t^n\th^{\mathrm{e}}\rangle.
\end{align*}
By Plancherel's formula, Young's convolution inequality and Lemma \ref{Lem:L1}, we have
\begin{align*}
    |J_{31}|\leq &C\Big(\|\langle\pa_x\rangle^{\f13}\sqrt{\mathcal{M}_t}\Lambda_t^n\pa_x\th^{\mathrm{e}}\|_{L^2}\|\hat{u}_{\neq}^{\mathrm{e},1}\|_{L^1}+\|\widehat{\pa_x\th^{\mathrm{e}}}\|_{L^1}\| \langle\pa_x\rangle^{\f13}\sqrt{\mathcal{M}_t}\Lambda_t^nu^{\mathrm{e},1}_{\neq}\|_{L^2}\big)\|\langle\pa_x\rangle^{\f13}\sqrt{\mathcal{M}_t}\Lambda_t^n\th^{\mathrm{e}}\|_{L^2}\\
    \leq&C(1+t)^{-1}\|\sqrt{\mathcal{M}_t}\Lambda_t^n \omega_{\neq}^{\mathrm{e}}\|_{L^2}\|\langle\pa_x\rangle^{\f13}\sqrt{\mathcal{M}_t}\Lambda_t^n\pa_x\th^{\mathrm{e}}\|_{L^2}\|\langle\pa_x\rangle^{\f13}\sqrt{\mathcal{M}_t}\Lambda_t^n\th^{\mathrm{e}}\|_{L^2}\\
    &+C\|\sqrt{\mathcal{M}_t}\Lambda_t^n\th^{\mathrm{e}}\|_{L^2}\| \sqrt{\mathcal{M}_t}\Lambda_t^n\omega^{\mathrm{e}}_{\neq}\|_{L^2}\|\langle\pa_x\rangle^{\f13}\sqrt{\mathcal{M}_t}\Lambda_t^n\th^{\mathrm{e}}\|_{L^2},
\end{align*}
and
\begin{align*}
    |J_{33}|\leq &C\big(\|\langle\pa_x\rangle^{\f13}\sqrt{\mathcal{M}_t}\Lambda_t^n\pa_y\th^{\mathrm{e}}\|_{L^2}\|\hat{u}^{\mathrm{e},2}_{\neq}\|_{L^1}+\|\widehat{\pa_y\th^{\mathrm{e}}}\|_{L^1}\| \langle\pa_x\rangle^{\f13}\sqrt{\mathcal{M}_t}\Lambda_t^nu^{\mathrm{e},2}_{\neq}\|_{L^2}\big)\|\langle\pa_x\rangle^{\f13}\sqrt{\mathcal{M}_t}\Lambda_t^n\th^{\mathrm{e}}\|_{L^2}\\
    \leq &C\|\langle\pa_x\rangle^{\f13}\sqrt{\mathcal{M}_t}\Lambda_t^n\pa_y\th^{\mathrm{e}}\|_{L^2}\|\sqrt{\mathcal{M}_t}\Lambda_t^nu^{\mathrm{e},2}_{\neq}\|_{L^2}\|\langle\pa_x\rangle^{\f13}\sqrt{\mathcal{M}_t}\Lambda_t^n\th^{\mathrm{e}}\|_{L^2}\\
    &+C\|\sqrt{\mathcal{M}_t}\Lambda_t^n\pa_y\th^{\mathrm{e}}\|_{L^2}\| \langle\pa_x\rangle^{\f13}\sqrt{\mathcal{M}_t}\Lambda_t^nu^{\mathrm{e},2}_{\neq}\|_{L^2}\|\langle\pa_x\rangle^{\f13}\sqrt{\mathcal{M}_t}\Lambda_t^n\th^{\mathrm{e}}\|_{L^2}.
\end{align*}
By taking $w=\omega_0^{\mathrm{e}}$, $f=g=\langle\pa_x\rangle^{\f13}\th^{\mathrm{e}}$ in Lemma \ref{Lem:u_0 comm-long}, we have
\begin{align*}
   |J_{32}|\leq C\|\sqrt{\mathcal{M}_t}\Lambda_t^n\omega_0^{\mathrm{e}}\|_{L^2}\|\langle\pa_x\rangle^{\f13}\sqrt{\mathcal{M}_t}\Lambda_t^n\langle\pa_x\rangle^{\f13}\th^{\mathrm{e}}\|_{L^2}^2.
\end{align*}
Here, if we treat $J_3$ in a similar way as in treating $J_2$, we can obtain a better estimate.
We decompose $J_4$ as $J_4=2(J_{41}+J_{42})$ with 
\begin{align*}
    J_{41}=\langle \langle\pa_x\rangle^{\f13}\sqrt{\mathcal{M}_t}\Lambda_t^n(u^{\mathrm{e},1}\pa_x\th^{\mathrm{i}}),
    \langle\pa_x\rangle^{\f13}\sqrt{\mathcal{M}_t}\Lambda_t^n\th^{\mathrm{e}}\rangle,\\
    J_{42}=\langle \langle\pa_x\rangle^{\f13}\sqrt{\mathcal{M}_t}\Lambda_t^n(u^{\mathrm{e},2}\pa_y\th^{\mathrm{i}}),
    \langle\pa_x\rangle^{\f13}\sqrt{\mathcal{M}_t}\Lambda_t^n\th^{\mathrm{e}}\rangle.
\end{align*}
For $J_{41}$, by Plancherel's formula and Young's convolution inequality, we have
\begin{align*}
    |J_{41}|\leq& C\|\sqrt{\mathcal{M}_t}\Lambda_t^n\omega^{\mathrm{e}}_{\neq}\|_{L^2}\|\langle\pa_x\rangle^{\f13}\sqrt{\mathcal{M}_t}\Lambda_t^n\pa_x\th_{\neq}^{\mathrm{i}}\|_{L^2}\|\langle\pa_x\rangle^{\f13}\sqrt{\mathcal{M}_t}\Lambda_t^n\th^{\mathrm{e}}\|_{L^2}\\
&+C\|\sqrt{\mathcal{M}_t}\Lambda_t^nu^{\mathrm{e},1}_{0}\|_{L^2}\|\langle\pa_x\rangle^{\f13}\sqrt{\mathcal{M}_t}\Lambda_t^n\pa_x\th^{\mathrm{i}}\|_{L^2}\|\langle\pa_x\rangle^{\f13}\sqrt{\mathcal{M}_t}\Lambda_t^n\th^{\mathrm{e}}_{\neq}\|_{L^2}.
\end{align*}
By Plancherel's formula and Young's convolution inequality, we know
\begin{align*}
    |J_{42}|
    \leq&\sum_{k,l}\iint \langle k\rangle^{\f13}\sqrt{\mathcal{M}_t(k,\xi)}\Lambda_t^n(k,\xi)|\hat{u}^{\mathrm{e},2}(l,\eta)||(\xi-\eta)\hat{\th}^{\mathrm{i}}(k-l,\xi-\eta)|\langle k\rangle^{\f13}\\
    &\times \sqrt{\mathcal{M}_t(k,\xi)}\Lambda_t^n(k,\xi)|\overline{\hat{\th}^{\mathrm{e}}(k,\xi)}|d\xi d\eta\\
    \leq&Ce^{\f12\delta_0\nu^{\f13}t}\sum_{k,l}\iint \langle k-l\rangle^{\f13}\Lambda_t^n(k-l,\xi-\eta)|\hat{u}^{\mathrm{e},2}(l,\eta)||(\xi-\eta)\hat{\th}^{\mathrm{i}}(k-l,\xi-\eta)|\langle k\rangle^{\f13}\\
    &\times\sqrt{\mathcal{M}_t(k,\xi)}\Lambda_t^n(k,\xi)|\overline{\hat{\th}^{\mathrm{e}}(k,\xi)}|d\xi d\eta\\
    &+Ce^{\f12\delta_0\nu^{\f13}t}\sum_{k,l}\iint \langle l\rangle^{\f13}\Lambda_t^n(l,\eta)|\hat{u}^{\mathrm{e},2}(l,\eta)||(\xi-\eta)\hat{\th}^{\mathrm{i}}(k-l,\xi-\eta)|\langle k\rangle^{\f13}\\
    &\times\sqrt{\mathcal{M}_t(k,\xi)}\Lambda_t^n(k,\xi)|\overline{\hat{\th}^{\mathrm{e}}(k,\xi)}|d\xi d\eta=:Ce^{\f12\delta_0\nu^{\f13}t}(J_{421}+J_{422}).
\end{align*}
By Young's inequality and taking $f=u_{\neq}^{\mathrm{e},2}$ in Lemma \ref{Lem:L1}, we have
\begin{align*}
    |J_{421}|\leq &C\|\hat{u}^{\mathrm{e},2}_{\neq}\|_{L^1}\|\langle\pa_x\rangle^{\f13}\Lambda_t^nD_y\th^{\mathrm{i}}\|_{L^2}\|\langle\pa_x\rangle^{\f13}\sqrt{\mathcal{M}_t}\Lambda_t^n\th^{\mathrm{e}}\|_{L^2}\\
    \leq&C(1+t)^{-1}\|\nabla \Lambda_t^nu^{\mathrm{e},2}_{\neq}\|_{L^2}\|\langle\pa_x\rangle^{\f13}\Lambda_t^nD_y\th^{\mathrm{i}}\|_{L^2}\|\langle\pa_x\rangle^{\f13}\sqrt{\mathcal{M}_t}\Lambda_t^n\th^{\mathrm{e}}\|_{L^2}.
\end{align*}
For $J_{422}$, by H\"older's inequlity, we know $J_{422}\leq J_{4221}^{\f12}J_{4222}^{\f12}$, with
\begin{align*}
    J_{4221}=&\sum_{k,l}\iint(l^2+\eta^2)\left|\Lambda_t^n(l,\eta)\hat{u}^{\mathrm{e},2}(l,\eta)\right|^2\Lambda_t^2(k-l,\xi-\eta)\left|(\xi-\eta)\hat{\th}^{\mathrm{i}}(k-l,\xi-\eta)\right|^2d\xi d\eta,\\
    J_{4222}=&\sum_{k,l\in\mathbb Z,l\neq0}\iint\f{\langle l\rangle^{\f23}\left|\langle k\rangle^{\f13}\sqrt{\mathcal{M}_t(k,\xi)}\Lambda_t^n(k,\xi)\overline{\hat{\th}^{\mathrm{e}}(k,\xi)}\right|^2}{(l^2+\eta^2)\Lambda_t^2(k-l,\xi-\eta)}d\xi d\eta.
\end{align*}
By Fubini's theorem, Plancherel's formula, and \eqref{eq:Mf}, we know
\begin{align*}
    |J_{4221}|\leq C\|\nabla \Lambda_t^nu^{\mathrm{e},2}\|_{L^2}^2\|\Lambda_tD_y\th^{\mathrm{i}}\|_{L^2}^2\leq C\nu^{-\f23}\|\nabla \Lambda_t^nu^{\mathrm{e},2}\|_{L^2}^2\|\sqrt{\mathcal{M}_t}\Lambda_t^n\th^{\mathrm{i}}\|_{L^2}^2.
\end{align*}
For $l\neq0$, applying \eqref{eq:as} with $a=|l|$, $s=1+|k-l|$, and $z=\xi+t(k-l)$, we get
\begin{align*}
    \int\f{\langle l\rangle^{\f23}d\eta}{(l^2+\eta^2)\Lambda_t^2(k-l,\xi-\eta)}\leq C\f{\langle l\rangle^{\f23}}{|l|(1+|k-l|)}\f{1+|k-l|+|l|}{(1+|k-l|+|l|)^2+(\xi+t(k-l))^2}.
\end{align*}
Then by the definition of $\Upsilon(t,k,\xi)$, we have
\begin{align*}
    |J_{4222}|\leq& C\sum_{k,l\in\mathbb Z,l\neq0}\int\f{\langle l\rangle^{\f23}(1+|k-l|+|l|)\left|\langle k\rangle^{\f13}\sqrt{\mathcal{M}_t(k,\xi)}\Lambda_t^n(k,\xi)\hat{\th}^{\mathrm{e}}(k,\xi)\right|^2}{|l|\big((1+|k-l|+|l|)^2+(\xi+t(k-l))^2\big)}d\xi\\
    \leq&C\|\langle\pa_x\rangle^{\f13}\sqrt{\Upsilon_t\mathcal{M}_t}\Lambda_t^n\th^{\mathrm{e}}\|_{L^2}^2.
\end{align*}
Then by the facts that $t>T_0=\nu^{-\f16}$ and \eqref{eq:Mneq}, we have
\begin{align*}
    |J_{42}|\leq& C \|\nabla \sqrt{\mathcal{M}_t}\Lambda_t^nu^{\mathrm{e},2}_{\neq}\|_{L^2}\Big(\nu^{\f16}\|\langle\pa_x\rangle^{\f13}\Lambda_t^nD_y\th^{\mathrm{i}}\|_{L^2}\|\langle\pa_x\rangle^{\f13}\sqrt{\mathcal{M}_t}\Lambda_t^n\th^{\mathrm{e}}\|_{L^2}\\
    &+\nu^{-\f13}\|\sqrt{\mathcal{M}_t}\Lambda_t^n\th^{\mathrm{i}}\|_{L^2}\|\langle\pa_x\rangle^{\f13}\sqrt{\Upsilon_t\mathcal{M}_t}\Lambda_t^n\th^{\mathrm{e}}\|_{L^2}\Big).
\end{align*}
Summing up, we have
\begin{align*}
    &\f{d}{dt}\|\langle\pa_x\rangle^{\f13}\sqrt{\mathcal{M}_t}\Lambda_t^n\th^{\mathrm{e}}\|_{L^2}^2+2\mathcal{CK}(\th^{\mathrm{e}})\\
    \leq&C\f{\ep_0^2\nu}{1+t^2}\delta\nu^{\f56}+C\delta\nu^{\f56}\nu^{-\f16}\mathcal{CK}^{\f12}(\omega^{\mathrm{i}})\nu^{-\f16}\mathcal{CK}^{\f12}(\th^{\mathrm{e}})+C\mathcal{CK}^{\f12}(\omega^{\mathrm{i}})\delta\nu^{\f56}\nu^{-\f16}\mathcal{CK}^{\f12}(\th^{\mathrm{e}})\\
    &+C\delta\nu^{\f13}\nu^{-\f13}\mathcal{CK}(\th^{\mathrm{e}})+C\delta\nu^{\f13}\nu^{\f23}\nu^{-\f12}\mathcal{CK}^{\f12}(\th^{\mathrm{e}})\nu^{-\f16}\mathcal{CK}^{\f12}(\th^{\mathrm{e}})\\
    &+C\f{\ep_0^2\nu^{\f12}}{1+t}\nu^{-\f12}\mathcal{CK}^{\f12}(\th^{\mathrm{e}})\nu^{-\f16}\mathcal{CK}^{\f12}(\th^{\mathrm{e}})+C\f{1}{1+t}\delta\nu^{\f12}\nu^{-\f16}\mathcal{CK}^{\f12}(\th^{\mathrm{e}})\nu^{-\f16}\mathcal{CK}^{\f12}(\th^{\mathrm{e}})\\
    &+C\delta\nu^{\f12}\nu^{-\f16}\mathcal{CK}^{\f12}(\th^{\mathrm{e}})\nu^{-\f16}\mathcal{CK}^{\f12}(\th^{\mathrm{e}})+C\nu^{-\f12}\mathcal{CK}^{\f12}(\th^{\mathrm{e}})\mathcal{CK}^{\f12}(\omega^{\mathrm{e}})\delta\nu^{\f56}\\
    &+C\nu^{-\f16}\mathcal{CK}^{\f12}(\th^{\mathrm{e}})\mathcal{CK}^{\f12}(\omega^{\mathrm{e}})\delta\nu^{\f56}+C\delta\nu^{\f12}\nu^{-\f13}\mathcal{CK}(\th^{\mathrm{e}})+C\delta\nu^{\f12}\nu^{-\f16}\mathcal{CK}^{\f12}(\th^{\mathrm{i}})\nu^{-\f16}\mathcal{CK}^{\f12}(\th^{\mathrm{e}})\\
    &+C\delta\nu^{\f12}\nu^{-\f16}\mathcal{CK}^{\f12}(\th^{\mathrm{i}})\nu^{-\f16}\mathcal{CK}^{\f12}(\th^{\mathrm{e}})+C\mathcal{CK}^{\f12}(\omega^{\mathrm{e}})\big(\nu^{\f16}\nu^{-\f12}\mathcal{CK}^{\f12}(\th^{\mathrm{i}})\delta\nu^{\f56}+\nu^{-\f13}\delta\nu^{\f23}\mathcal{CK}^{\f12}(\th^{\mathrm{e}})\big)\\
    \leq& C_{\delta_0}\f{\ep_0^2\delta\nu^{\f{11}{6}}}{1+t^2}+C_{\delta_0}(\delta+\ep_0^2)\mathcal{CK}(\th^{\mathrm{e}})+C_{\delta_0}\delta\nu^{\f23}\mathcal{CK}(\omega^{\mathrm{e}})+C_{\delta_0}\delta\nu\mathcal{CK}(\omega^{\mathrm{i}})+C_{\delta_0}\delta\nu^{\f13}\mathcal{CK}(\th^{\mathrm{i}}).
\end{align*}
Then combined with \eqref{ineq:w}, we obtain the results.
\end{proof}

We now conclude this section by summarizing the estimates of $\omega^{\mathrm{e}}$ and $\th^{\mathrm{e}}$. By the standard bootstrap argument and the estimates we obtained in Section \ref{sec-i} about $\omega^{\mathrm{i}}$ and $\th^{\mathrm{i}}$, we have for $t\leq \nu^{-\f16}$ that
\begin{align*}
    &\|\Lambda_t^n\omega^{\mathrm{e}}\|_{L^2}\leq \delta\nu^{\frac{2}{3}},\quad
    \nu^{\frac{1}{6}}\|\partial_x\Lambda_t^n\theta^{\mathrm{e}}(t) \|_{L^2}+\| \Lambda_t^n\theta^{\mathrm{e}}(t) \|_{L^2}\leq \delta \nu,
\end{align*}
and for $t>\nu^{-\f16}$,
\begin{align*}
    \|\Lambda_t^n\sqrt{\mathcal{M}_t}\omega^{\mathrm{e}}(t)\|_{L^{\infty}_tL^2}\leq C\epsilon_0\nu^{\f12},\quad \|\langle \partial_x\rangle^{\frac{1}{3}} \sqrt{\mathcal{M}_t}\Lambda_t^n\theta^{\mathrm{e}}(t) \|_{L^{\infty}_tL^2}\leq C\epsilon_0\nu^{\f56}.
\end{align*}
Theorem \ref{main-thm} then follows directly by adding the estimates of $(\omega^{\mathrm{i}},\theta^{\mathrm{i}})$ in section \ref{sec-i} together with the estimate of $(\omega^{\mathrm{e}},\theta^{\mathrm{e}})$, and applying Lemma \ref{Lem: ID}.

\bibliographystyle{abbrv.bst} 
\bibliography{references.bib}

\end{document}